\documentclass[11pt]{article} 
  \usepackage{amsmath}
    \usepackage{amsthm}
    \usepackage{amssymb}
   \usepackage[dvipsnames]{xcolor}  
   \usepackage{graphicx}
    \usepackage{pdfsync} 
\def\corON{}

\def\corO{}
\def\corD{}
\def\corDB{}
\def\corOF{}
\def\corDBa{}
\def\corDBb{}

\newtheorem{theorem}{Theorem}[section]
\newtheorem{lemma}[theorem]{Lemma}%[section]
\newtheorem{corollary}[theorem]{Corollary}%[section]
\newtheorem{proposition}[theorem]{Proposition}%[section]
\newtheorem{remark}[theorem]{Remark}%[section]
\newcommand{\eqnsection}{
   \renewcommand{\theequation}{\thesection.\arabic{equation}}
   \makeatletter
   \csname @addtoreset\endcsname{equation}{section} 
   \makeatother}
   \def\nc{}

\def \be{\begin{equation}}
\def \ee{\end{equation}}
\def \bt{\begin{theorem}} 
\def \et{\end{theorem}}
\def \bl{\begin{lemma}} 
\def \el{\end{lemma}}
\def \bea{\begin{eqnarray}}
\def \eea{\end{eqnarray}}
\def \bas{\begin{eqnarray*}}
\def \eas{\end{eqnarray*}}
\def \hit{\corDB{{\tau}}}

\renewcommand{\eqref}[1]{(\ref{#1})} % Because \eqref gives non-italic digits even in theorem enviroment, while \ref does not, and we use both

\def \al{\alpha}
\def \bb{\beta}
\def \ga{\gamma}

\def \de{\delta}

\def \la{\lambda}

\def \si{\sigma}

\def \ff{\infty}
\def \wh{\widehat}
\def \wt{\widetilde}

\def \FF{{\cal F}}

\def \ll{{\mathbf L}}

\def\b1{\mathbf 1}

\def \({\left(}
\def \){\right)}

\def \lc{\left\{}
\def \rc{\right\}}

\def \nn{\nonumber}

\def \bc{\begin{center} }
\def \ec{\end{center} }
\def \bs{\begin{slide} }
\def \es{\end{slide} }

\def\square{{\vcenter{\vbox{\hrule height.3pt
        \hbox{\vrule width.3pt height5pt \kern5pt
           \vrule width.3pt}
        \hrule height.3pt}}}}
\def\qed{{\hfill $\square$ \bigskip}}

\eqnsection

\begin{document}
	
\title{Barrier estimates for a critical Galton--Watson process and the cover time
of the binary tree}

\author{David Belius, Jay Rosen\footnote{Partially supported by grants from the
NSF and from the Simons foundation}, Ofer Zeitouni
\footnote{Partially supported by
the ERC advanced grant LogCorrelatedFields}}
%\subjclass[2010]{60J80}
%\keywords{Barrier estimates. Galton-Watson process. Coupling. Random walk.}
\maketitle
\begin{abstract}
  \corON{
For the
  critical Galton--Watson process with
  geometric offspring distributions we provide sharp barrier estimates  for barriers which are (small) 
  perturbations of linear barriers. These are useful in analyzing the cover
  time of finite  graphs in the critical regime by random walk, and the 
  Brownian cover times of 
  compact two dimensional manifolds.  As an application
  of the barrier estimates, we prove that if $C_L$ denotes 
  the cover time of the binary tree
  of depth $L$ by simple  walk, then
  $\sqrt{C_L/2^{L+1}} -\sqrt{2\log 2} L+\log L/\sqrt{2\log 2}$ is tight. The
  latter improves results of Aldous (1991),
  Bramson and Zeitouni (2009) and 
  Ding and Zeitouni (2012). In a subsequent 
article we use these barrier estimates to prove tightness of the Brownian cover time 
for the two-dimensional sphere.
}
\end{abstract}
\section{Introduction and statement of main results}

Let $P_{n}$ be the law of the critical Galton--Watson process $\left(T_{l}\right)_{l\ge0}$
with inital population $T_{0}=n$ and geometric offspring distribution.
For any $a,b,L$ let
\[
f_{a, b}\left(l;L\right)=a +(b-a)\frac{l}{L},
\]
be the line interpolating $a$ and $b$ over the interval $\left[0,L\right]$.
Abbreviate $l_{L}=l\wedge\left(L-l\right)$. Let $\mathbb{N}=\left\{ 1,2,\ldots\right\} $ \corDB{and $\mathbb{Z}^+ = \left\{0,1,2,\ldots \right\}$}. \corON{For $y,\delta\geq 0$, let $H_{y,\delta}=[y,y+\delta]$ and
set $H_y=H_{y,1}$.} \corDB{The main result of this article is the following barrier estimates for the process $T_l$.}

%\begin{figure}\label{fig: barrier events}
%	\centering
%	%\includegraphics[width=1.0\textwidth]{2016-11-08_figure_2.png}
%	\includegraphics[width=0.7\textwidth]{2017-01-20figurenostraightline.pdf}
%	\vspace{-2in}
%	\caption{Illustration of barrier events in Theorem 
%		\ref{thm: GW Barrier}. Depicted is the event in part b).}
%\end{figure}

%\begin{figure}\label{fig: barrier events}
%	\centering
%	\begin{minipage}[b]{0.4\textwidth}
%		\centering
%		\includegraphics[width=1.5\textwidth]{2017-01-20figurenostraightline_UB.pdf}
%		\caption{A figure}
%	\end{minipage}
%	\hfill
%	\begin{minipage}[b]{0.4\textwidth}
%		\centering
%		\includegraphics[width=1.5\textwidth]{2017-01-20figurenostraightline_LB.pdf}
%		\caption{A figure}
%	\end{minipage}
%%	\vspace{-2in}
%%	\caption{Illustration of barrier event in Theorem \ref{thm: GW Barrier} a).}
%\end{figure}

\begin{figure}\label{fig: barrier events UB}
	\centering
	\includegraphics[width=0.7\textwidth]{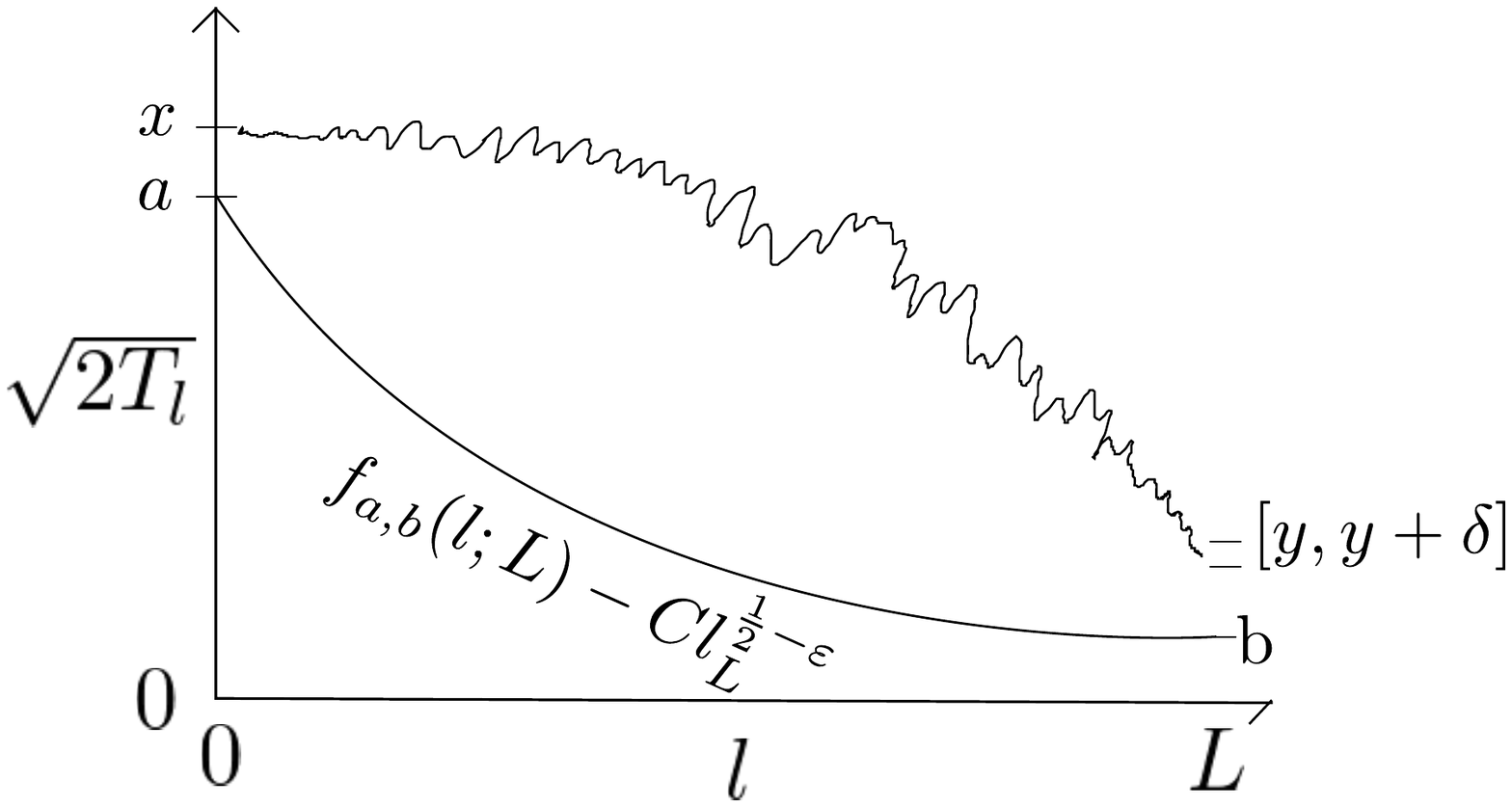}
	\vspace{-2in}
	\caption{Illustration of barrier event in Theorem \ref{thm: GW Barrier} a).}
\end{figure}

\begin{figure}\label{fig: barrier events LB}
	\centering
	\includegraphics[width=0.7\textwidth]{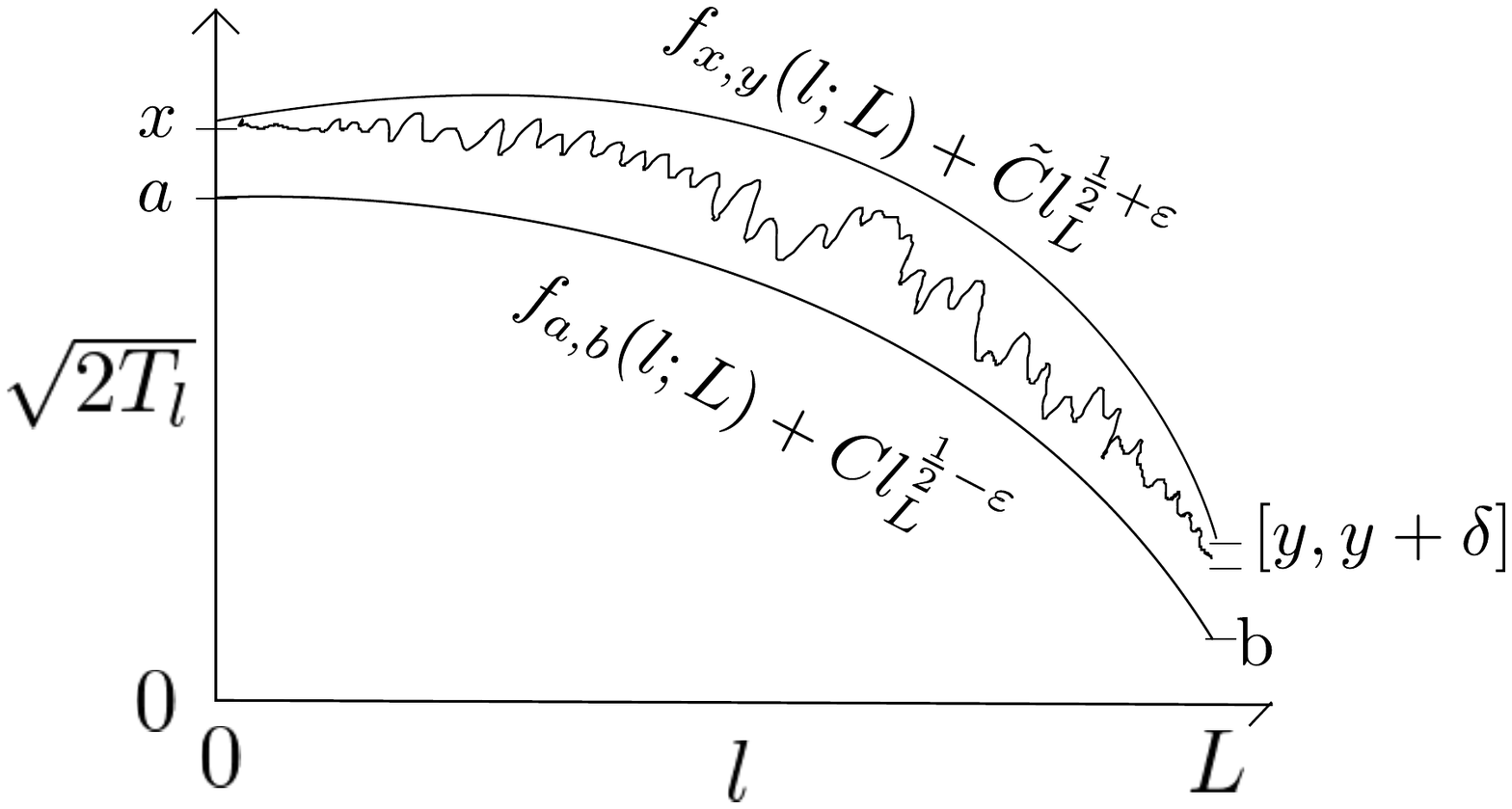}
	\vspace{-2in}
	\caption{Illustration of barrier event in Theorem \ref{thm: GW Barrier} b).}
\end{figure}

\bt
\label{thm: GW Barrier}
\corON{a)} For all fixed $\delta, C>0$, \corON{$\eta>1$} and 
$\varepsilon\in\left(0,\frac{1}{2}\right)$ we have, \corON{
  %for all $L$ 
%large enough and
uniformly
in} $\sqrt{2}\le x, y\le \eta L$ 
\corON{
such that ${x^{2}}/{2}\in\mathbb{N}$, any $0\le a\le x, \,\,0\le b\le y$,}
that
%and $L$ large enough that
\begin{eqnarray}
&&P_{x^{2}/2}\left(f_{a, b}\left(l;L\right)-Cl_{L}^{\frac{1}{2}-\varepsilon}\le\sqrt{2T_{l}},l=1,\ldots,L-1,\sqrt{2T_{L}}\in
  \corON{H_{y,\delta}}\right) \nonumber\\
  %\left[y,y+\delta\right]\right)\\
&&\;\le c\frac{\left(1+x-a\right)\left(1+y-b\right)}{L} \sqrt{
  \corON{\frac{x}{yL}}}e^{-\frac{\left(x-y\right)^{2}}{2L}}.
\label{eq: GW curved barrier upper bound}
%\label{eq: GW straight barrier upper bound}
\end{eqnarray}
If in addition $a\geq {L}/{\eta}$, similar bounds hold for $y\leq \sqrt{2}$,
with \corON{$\sqrt{\frac{x}{yL}}$} replaced by $1$.
%\begin{equation}
%\begin{array}{l}
%P_{x^{2}/2}\left(f_{a, b}\left(l;L\right)-Cl_{L}^{\frac{1}{2}-\varepsilon}\le%\sqrt{2T_{l}},l=1,\ldots,L-1,\sqrt{2T_{L}}\in\left[y,y+\delta\right]\right)\\
%\le c\frac{\left(1+x-a\right)\left(1+y-b\right)}{L}\times\left(\sqrt{\frac{x/y}{L}}\wedge1\right)e^{-\frac{\left(x-%y\right)^{2}}{2L}}.
%\end{array}\label{eq: GW curved barrier upper bound}
%\end{equation}

\corON{b)} \nc For any  $\tilde{C}\geq 2C+2\delta+\eta+ \sqrt{2} $, if, \corON{in addition to the conditions in part a),} 
we also have  $\left(1+x-a\right)\left(1+y-b\right)\le \eta L$, $\max (ab, |a-b|)\geq  L/\eta$
and $\left[y,y+\delta\right]\cap\sqrt{2\mathbb{Z}^+}\ne\emptyset$ then
\begin{eqnarray*}
  %\begin{equation}
%\begin{array}{l}
&&P_{x^{2}/2}\left(\begin{array}{r}
f_{a, b}\left(l;L\right)+Cl_{L}^{\frac{1}{2}-\varepsilon}\le\sqrt{2T_{l}}\le f_{x, y}\left(l;L\right)+\tilde{C}l_{L}^{\frac{1}{2}+\varepsilon},l=1,\ldots,L-1,
\nonumber \\
\sqrt{2T_{L}}\in \corON{H_{y,\delta}}\nonumber \\
%\left[y,y+\delta\right]
\end{array}\right)\\
&& \;\ge c\frac{\left(1+x-a\right)\left(1+y-b\right)}{L}\times\left(\sqrt{
  \corON{\frac{x}{yL}}}\wedge1\right)e^{-\frac{\left(x-y\right)^{2}}{2L}},
%\end{array}
\label{eq: gw lower bound}
\end{eqnarray*}
\corON{and the estimate is uniform in such $x,y,a,b$ and all $L$.}
\et

%\begin{remark}
%In general, the first inequality in  (\ref{eq: GW straight barrier upper bound}) is strict.
%\end{remark}
%
\begin{remark}
\corDBb{a)} In this paper constants, 
\corON{whose value may change from occurence to occurence,} that depend 
at most on \corON{$C,\delta,\eta$} and $\varepsilon$,
are denoted by $c$. The notation $a\asymp b$ means that $a\le c\cdot b$
and $b\le c\cdot a$.\\
\corDBb{b)} When $\delta<\sqrt{2}$   and $y=0$ the terminal condition $\sqrt{2T_{L}}\in
\corON{y\in H_{y,\delta}}$
%\left[y,y+\delta\right]$
is equivalent to $T_{L}=0$. \corDBb{Less precise barrier estimates for the process $T_l$ conditioned on $T_L=0$ appear in \cite[Proposition 7.1]{BK}.}
\end{remark}
 \corOF{We will put Theorem \ref{thm: GW Barrier} in the context of 
estimates for Bessel processes, see Proposition 
\ref{prop: Almost LCLT} below.} 
%and reformulate it in Proposition 
%\ref{thm: GW Barrier concise} below. 
Before doing so, we emphasize that 
the main application of Theorem \ref{thm: GW Barrier} is in an upcoming
paper by the authors which proves tightness of the (centered) square-root of
the Brownian
cover 
time of the two dimensional 
sphere. In the current paper, we illustrate the use
of Theorem \ref{thm: GW Barrier} by \corOF{presenting a}
%using it to obtain a 
quick proof of a  similar result for
the cover time of the \corO{binary}
tree. Let $\mathcal{T}_{L}$ be the tree of
depth $L$, with a root of degree one attached to the top. \corO{(Formally,
begin with a binary tree rooted at vertex \corDB{$o$}, and attach to it a vertex \corDB{$\rho$},
\corON{the root of $\mathcal{T}_L$}, connected by an edge to \corDB{$o$}.)}
%Our tree
\corOF{The tree $\mathcal{T}_L$}
has $2^{L+1}$ vertices and \corON{$2^{L}+1$} leaves \corDBa{(including $\rho$)}. Let $\mathbb{P}$ be
the law of discrete time simple random walk $\left(X_{n}\right)_{n\ge0}$ on $\mathcal{T}_{L}$ starting at the root. Let $\hit_{y},y\in\mathcal{T}_{L}$,
be the hitting time of the vertex $y$ by the random walk. 
The cover time
\[
C_{L}=\max_{y\in\mathcal{T}_{L}}\hit_{y}=\max_{y\in\mathcal{T}_{L},y\text{ is a leaf}}\hit_{y}
\]
is the first time the random walk has visited every vertex of $\mathcal{T}_{L}$.
We prove the following estimate.
 
\bt
\label{thm: tree cover main} \corDB{There exist constants $c$ such that }for all $x>0$,
\begin{equation}
	\limsup_{L\to\infty}\mathbb{P}\left(\sqrt{\frac{C_{L}}{2^{L+1}}}\ge\sqrt{2\log2}\,\,L-\frac{1}{\sqrt{2\log2}}\log L+x\right)\le cxe^{-x \sqrt{2\log2}},\label{eq: upper tail bound}
\end{equation}
%\corDB{and
\begin{equation}
	\liminf_{L\to\infty}\mathbb{P}\left(\sqrt{\frac{C_{L}}{2^{L+1}}}\ge\sqrt{2\log2}\,\,L-\frac{1}{\sqrt{2\log2}}\log L+x\right) \ge cxe^{- x \sqrt{2\log2}},\label{eq: upper tail lower bound}
\end{equation}
and
\begin{equation}
\limsup_{L\to\infty}\mathbb{P}\left(\sqrt{\frac{C_{L}}{2^{L+1}}}\le\sqrt{2\log2}\,\,L-\frac{1}{\sqrt{2\log2}}\log L-x\right)\le e^{-cx}.\label{eq: lower tail bound}
\end{equation}
\et
In particular, \corOF{Theorem \ref{thm: tree cover main} shows that} 
\begin{equation}\label{eq: tightness informal}
\sqrt{\frac{C_{L}}{2^{L+1}}}=\sqrt{2\log2}\,\,L-\frac{1}{\sqrt{2\log2}}\log L+O\left(1\right),
\end{equation}
\corOF{that is, tightness of a centered, scaled version of the square root of the
cover time.}
\corOF{
Equivalently one could state tightness directly in terms of a centered and scaled version of the cover time itself, as}
%without the square root scaling as
$$
\corOF{\frac{C_L}{2^{L+1} L} =  2 \left( \log 2\cdot  L - \log L + O(1)
\right),}
$$
and the corresponding tail bounds can also be written in a similar way.
\corDB{The statement
  \eqref{eq: tightness informal}} improves on the estimate from \cite{DingZeitouni-ASharpEstimateForCoverTimesOnBinaryTrees},
which has $O\left(\left(\log\log L\right)^{8}\right)$ in place of
$O\left(1\right)$. \corON{(Earlier results of Aldous \cite{Aldous} give 
the leading order\\ $\sqrt{2\log 2}\, L(1+o(1))$.)} 
\corON{Theorem \ref{thm: tree cover main}}
also provides a new proof that after appropriate
centering, $\sqrt{C_{L}/2^{L+1}}$ is tight, a result proven in 
\cite{BramsonZeitouni}
using a recursion (in a way that avoids computing the centering
term).

%\corOF{In recent years, 
%  cover times of graphs by random walk and of certain manifolds by 
%  Brownian motion
%  were related to the extrema of Gaussian free fields
%  \cite{DLP, Ding}.

%In the particular cases of
%homogeneous
%trees and of the torus (as representative of a two dimensional 
%homogeneous manifold),
%the cover times were
%related to the extrema of Gaussian log-correlated random fields, }
 %a universality class of which branching random walk is the canonical 
 %representative 
 %\cite{DingZeitouni-ASharpEstimateForCoverTimesOnBinaryTrees,BK}. 
 %The logarithmic correction term in 
 %\eqref{eq: tightness informal}, whose form was proven in 
 %\cite{DingZeitouni-ASharpEstimateForCoverTimesOnBinaryTrees}, is 
 %\corOF{obtained from}
 %the universal correction term for fields in this class. 
 \corDBa{In light of recent works 
   \cite{DingZeitouni-ASharpEstimateForCoverTimesOnBinaryTrees,BK}, it appears 
   that cover times of homogeneous trees and of two dimensional graphs 
   or manifolds are related to the extrema \corOF{of certain critical hierarchical 
   random fields},
   %\corOF{of log-correlated} 
   %random fields, 
   a universality class \corOF{which contains logarithmically 
   correlated fields and branching random walks}\footnote{
   %and
   %of which branching random walk, not necessarily with Gaussian tails, is the 
   %canonical representative
   \corDBb{More generally, isomorphism theorems have been used to show that for 
   any finite connected graph \corOF{for which hitting times are asymptotically
   shorter than cover times}, the \corOF{
     %square root of the 
   cover time }
   divided by the number of edges is of the same order 
   as the \corOF{square of the} maximum of the Gaussian Free Field on the 
   same graph \cite{DLP}; for arbitrary trees  or for bounded degree graphs,
   %and with additional assumptions \corOF{which are satisfied for arbitrary trees or bounded degree graphs,} 
   they match 
   to leading order \cite{Ding}.}}.
   The logarithmic correction term in \eqref{eq: tightness informal}, whose 
   form was proven 
   in \cite{DingZeitouni-ASharpEstimateForCoverTimesOnBinaryTrees}, is, up to constant multiple,  
   the universal correction term\footnote{\corDBb{The constant multiple is determined by tail estimates.
   		Differing tail estimates for local times and related Gaussian fields
   		lead to differing constants for cover times and the maximum of
   		such Gaussian fields, as explained in \cite[Section 1.2]{BK}.}} for fields in \corOF{this}
 universality class.
% and the constant multiple is determined by %\corOF{pointwise tail 
% estimates of the field} \corDBa{(differing tail estimates %for local times and related Gaussian fields lead to %differing constant for cover times and these Gaussian %fields,  see Section 1.2 \corOF{\cite{BK}}.}
%
%That the logarithmic correction term is 
% different in the cases of cover time and of 
%branching random walks can be understood as a consequence of differing pointwise tail estimates of the different %fields, see \corDBa{Section 1.2} \corOF{\cite{BK}}.
%% that is Gaussian vs. 
%% square-root of exponential.} 
 That the minimum (or maximum) is tight after centering by the leading term together with this logarithmic  correction term is another \corOF{conjectured}
 universal feature of these fields, as is the decay $x e^{-cx}$, for some $c$, of the right tail, which we verify in \eqref{eq: upper tail bound} and \eqref{eq: upper tail lower bound}. \corDBa{(By contrast, the left tail for which we have only the rough bound \eqref{eq: lower tail bound}, is not \corOF{expected to be
 universal, and in general is not even of exponential form.)}} Our approach builds on previous works on branching random
walk \cite{Bramson1ConvergenceofSolutionsOfKolmogorovEqn}, \corON{\cite{ABR}},
\cite{ABKGenealogy},
\corON{\cite{Aidekon},}
\cite{BramsonDingZeitouni-ConvergenceinLawOfTheMaxOfNonLAtticeBRW}
and cover times \cite{DemboPeresEtAl-CoverTimesforBMandRWin2D}, \cite{DingZeitouni-ASharpEstimateForCoverTimesOnBinaryTrees},
\cite{BK}. More precisely, we use a second moment method
with a truncation involving the process of
\corDBb{discrete} edge local times staying
above certain barriers. The barrier estimates that are the main results
of this paper are a crucial technical input.

\corDBb{The main step of the proof of the cover time result is the analysis of the {\it minimum} of the discrete local times among the leaves at the time a certain local time is reached at the root. For the {\it maximum} of continuous local times on the leaves more precise results have been obtained \cite{Abe}.}

\corON{The intuition behind Theorem
  \ref{thm: GW Barrier} is that
the process $l \to \sqrt{2 T_l}$ behaves like a Bessel-0 process,
\corDB{i.e. a process $Y_t$ which satisfies the SDE
\begin{equation}\label{eq: bessel SDE}
	dY_{t}=dW_{t}-\frac{1}{2Y_{t}}dt,
\end{equation}
for a Brownian motion $W_t$ (the drift is only significant if 
$Y_t$ is close to zero,
otherwise the process $Y_t$ 
behaves like a Brownian motion).} The Galton--Watson process $T_l$
can be thought of as 
a discrete version of a squared Bessel-0 process.
%, which by the Ray--Knight theorem is the law of the local times of Brownian motion.
\corDB{Indeed}, if 
$f_{l,j},j \ge 1,$ are the number of offspring
of each individual in generation $l$, letting $g_{l,j}=f_{l,j}-1$
we can also write $\sqrt{2 T_{l+1}}$ as
\begin{equation}
\sqrt{2 T_{l+1}} =  \sqrt{2 T_l}  \sqrt{1+\frac{1}{T_l}\sum_{j=1}^{T_l}
g_{l,j}}.
\end{equation}
When $T_l$ is large one can Taylor expand the square root to obtain that
the increment
$\sqrt{2 T_{l+1}} - \sqrt{2 T_{l}}$ equals
\begin{equation}\label{eq: sum}
	\frac{1}{\sqrt{2T_l}}\sum_{j=1}^{T_l}g_{l,j} - \corON{\frac{1}{2 
	  \sqrt{2 T_l}} }
	%+ O\left( \frac{1}{\sqrt{2T_l}}\left( \frac{1}{2T_l}\sum_{j=1}^{T_l}g_{l} \right)^2 \right)
\end{equation}
plus terms that can be shown to be negligible.
The distribution of the normalized sum in \eqref{eq: sum} will be close to Gaussian by the central limit theorem (the $g_{l,j}$ are independent, have mean zero, and are independent of $T_l$). \corDB{Thus if we let $Z_l = \sqrt{2 T_l}$ we can informally write
$$
Z_{l+1} - Z_l \approx N_l - \frac{1}{2 Z_l },
$$
for approximately Gaussian independent $N_l$, making the heuristic link to \eqref{eq: bessel SDE} apparent.}
% The extra drift term in \eqref{eq: sum} is what makes this a good approximation to a Bessel process, \corDB{cf. \eqref{eq: bessel SDE}}.
A proof of \corON{results similar to} Theorem 
\ref{thm: GW Barrier}, that could be extended to other Galton--Watson processes,
could be provided\footnote{and is available from the authors} by following these ideas. Instead, in this paper, we use exact equalities in law to provide 
a shorter proof.
}

\corON{ Theorem  \ref{thm: GW Barrier}}
should be understood in the light of the following precise
large deviation estimate for $\sqrt{2T_{L}}$.

\begin{proposition}
  \label{prop: Almost LCLT}For all fixed $\delta>0$ \corON{and $\eta>1$}
  we have uniformly
in $\sqrt{2}\le x,y\le
\corON{\eta L}$ such that $x^{2}/2$ is an
integer that
\begin{equation}
  P_{x^{2}/2}\left(\sqrt{2T_{L}}\in \corON{H_{y,\delta}}
  %\left[y,y+\delta\right]i
  \right)\leq c\sqrt{\corON{\frac{x}{yL}}}e^{-\frac{\left(x-y\right)^{2}}{2L}}.\label{eq: Almost LCLT}
\end{equation}
If, in addition, \corDB{$\left[y,y+\delta\right]\cap\sqrt{2\mathbb{Z}^+}\ne\emptyset$} and $\corON{ L/\eta}\leq xy $, then the corresponding lower  bound
\corDB{
\begin{equation}
P_{x^{2}/2}\left(\sqrt{2T_{L}}\in H_{y,\delta}
%\left[y,y+\delta\right]i
\right) \geq
c\sqrt{\frac{x}{yL}}e^{-\frac{\left(x-y\right)^{2}}{2L}},
%\le 
%P_{x^{2}/2}\left(\sqrt{2T_{L}}\in H_{y,\delta}
%%\left[y,y+\delta\right]i
%\right),
\label{eq: Almost LCLT LB}
\end{equation}
}
also holds.
Also, for any $0<x\le\corON{\eta L}$,
%\frac{L}{\delta}$,
\begin{equation}
P_{x^{2}/2}\left(\sqrt{2T_{L}}=0\right)\asymp e^{-\frac{x^{2}}{2L}}.\label{eq: LCLT zero}
\end{equation}
\end{proposition}

The bounds (\ref{eq: Almost LCLT})-(\ref{eq: LCLT zero}) are the
same as the ones satisfied by $Y_{L}$ if the process
$Y_{t}$, under $P_{x}^{Y}$,
is a Bessel process of dimension zero\corDB{, since
\begin{equation}
P_{x}^{Y}\left(Y_{L}\in\cdot\right)=\delta_{0}e^{-\frac{x^{2}}{2L}}+1_{\left(0,\infty\right)}\frac{x}{L}e^{-\frac{x^{2}+y^{2}}{2L}}I_{1}\left(\frac{xy}{L}\right)dy,\label{eq: Bessel Proc Marginal}
\end{equation}
see above Lemma \ref{lem: Bessel Barrier}. Furthermore \corDBa{e.g.} Theorem \ref{thm: GW Barrier} a) can be written as
\begin{eqnarray}
&&P_{x^{2}/2}\left(f_{a, b}\left(l;L\right)-Cl_{L}^{\frac{1}{2}-\varepsilon}\le\sqrt{2T_{l}},l=1,\ldots,L-1 \big| \sqrt{2T_{L}}\in
\corON{H_{y,\delta}}\right) \nonumber\\
%\left[y,y+\delta\right]\right)\\
&&\le c\frac{\left(1+x-a\right)\left(1+y-b\right)}{L}.
\label{eq: GW bridge upper bound}
%\label{eq: GW straight barrier upper bound}
\end{eqnarray}
The probability that a Brownian bridge \corDBa{in the time interval $[0,L]$} stays above a linear barrier \corDBa{(or a small perturbation thereof) during $[1,L-1]$ as in the event in \eqref{eq: GW bridge upper bound}}, when it starts at distance $x-a$ from the line and ends at distance $y-b$ from it is of the order of the right-hand side of \eqref{eq: GW bridge upper bound} \corDBa{(see \eqref{eq: Brownian Bridge Barrier} and \eqref{eq: Brownian Barrier} below)}. Our Theorem \ref{thm: GW Barrier} can thus be thought of as a Galton--Watson process 
version of barrier results for the Brownian bridge. Heuristically, it arises from
\corOF{
  approximating $\sqrt{2 T_l}$ conditioned on its} end point by a Brownian bridge.
}

  \corON{
In the rest of the paper, we prove 
%Proposition  \ref{thm: GW Barrier concise}
%and 
\corOF{Theorems \ref{thm: GW Barrier}  and}
\ref{thm: tree cover main}. We begin by proving,
in Section \ref{sec-2}, barrier estimates for 0-dimensional Bessel processes.
In Section \ref{sec-3}, we show that traversal counts and local times of
a random walk on $\mathbb{Z}^+$
give a Markovian structure closely related to the sampling of a 
Bessel-0 process. In Section \ref{sec-4}, we use the latter structure 
to transfer barrier estimates for Bessel-0  processes to the setting
of Theorem \ref{thm: GW Barrier}.
%Proposition  \ref{thm: GW Barrier concise}. 
Finally, in
Section \ref{sec-5} we use a first/second moment method, together
with the barrier estimates of Theorem  \ref{thm: GW Barrier}, to 
obtain Theorem \ref{thm: tree cover main}.}

\section{Barrier estimates for the $0$-dimensional\\ Bessel process}
\label{sec-2}

We will derive the barrier estimate in 
\corOF{Theorem \ref{thm: GW Barrier}}
%Proposition \ref{thm: GW Barrier concise}
from similar results for a Brownian motion. Let $P_{x}^{W}$ be the
law of a Brownian motion $W_{t},t\ge0,$ starting at $x$, and let
$P_{x}^{W}\left(\cdot|W_{L}=y\right)$
be the law of a Brownian bridge starting at $x$ and ending at $y$
at time $L$. The probability that   a Brownian bridge stays above a linear
barrier is explicit: If   $a\le x$ and $b\le y$ then by the reflection principle, see e.g. \cite[Lemma 2.2]{Bramson1ConvergenceofSolutionsOfKolmogorovEqn}, 
\begin{equation}
P_{x}^{W}\left(W_{l}\ge f_{a, b}\left(l;L\right),l\in\left[0,L\right]|W_{L}=y\right)=1-\exp\left(-2\frac{\left(x-a\right)\left(y-b\right)}{L}\right).\label{eq: Brownian bridge barrier exact}
\end{equation}
If also $\left(x-a\right)\left(y-b\right)=O\left(L\right)$ this implies  
\[
P_{x}^{W}\left(W_{l}\ge f_{a, b}\left(l;L\right),l\in\left[0,L\right]|W_{L}=y\right)\asymp \frac{\left(x-a\right)\left(y-b\right)}{L}.
\]
and by conditioning on $W_{1}$ and $W_{L-1}$ one easily derives
that
\begin{equation}
P_{x}^{W}\left(W_{l}\ge f_{a, b}\left(l;L\right),l\in\left[1,L-1\right]|W_{L}=y\right)\asymp\frac{\left(1+x-a\right)\left(1+y-b\right)}{L}.\label{eq: Brownian Bridge Barrier}
\end{equation}
The next lemma shows that the probability has the same order of magnitude
if a ``bump'' is added or subtracted from the straight line.

\bl
\label{lem: Brownian Barrier-1}For all fixed $\varepsilon\in\left(0,\frac{1}{2}\right),\delta>0$, \corON{$\eta>1$}, $\tilde{C}\geq C+\de>0,$
one has uniformly in $a\le x,\,\,b\le y,$ $\left(x-a\right)\left(y-b\right)\le
\corON{\eta L}$,
%\frac{L}{\delta}$,
$\left|x-y\right|\le
\corON{\eta L}$,
%\frac{L}{\delta}$ 
and $L$ large enough,
\begin{equation}
\begin{array}{l}
P_{x}^{W}\left(f_{a, b}\left(l;L\right)+Cl_{L}^{\frac{1}{2}-\varepsilon}\le W_{l}\le f_{x, y}\left(l;L\right)+\tilde{C}l_{L}^{\frac{1}{2}+\varepsilon},l\in\left[1,L-1\right],W_{L}\in
\corON{H_{y,\delta}}\right)\\
%\left[y,y+\delta\right]\right)\\
\asymp P_{x}^{W}\left(f_{a, b}\left(l;L\right)-Cl_{L}^{\frac{1}{2}-\varepsilon}\le W_{l},l\in\left[1,L-1\right],W_{L}\in
%\left[y,y+\delta\right]\right)\\
\corON{H_{y,\delta}}\right)\\
\asymp\frac{\left(1+x-a\right)\left(1+y-b\right)}{L}\frac{1}{\sqrt{L}}e^{-\frac{\left(x-y\right)^{2}}{2L}}.
\end{array}\label{eq: Brownian Barrier}
\end{equation}
\el
\begin{remark}\label{rem-ub}
The condition $\left(x-a\right)\left(y-b\right)\le \corON{\eta L}$
%{L}/{\delta}$ 
is not necessary for the upper bounds. One can simply drop the barrier. If we eliminate the condition $\left|x-y\right|\le
\corON{\eta L}$
%{L}/{\delta}$ 
then the last line in (\ref{eq: Brownian Barrier}) would become
\begin{equation}
\frac{\left(1+x-a\right)\left(1+y-b\right)}{L}\frac{1}{\sqrt{L}}\sup_{z\in
%\left[y,y+\delta\right]
\corON{H_{y,\delta}}}e^{-\frac{\left(x-z\right)^{2}}{2L}}.\label{expcomp}
\end{equation}
for the upper bound, and   $\inf$ instead of $\sup$ for the lower bound. 
The condition $\left|x-y\right|\le
\corON{\eta L}$
%{L}/{\delta}$ 
is only used to guarantee that 
the \corON{supremum in \eqref{expcomp}} has the same order as $e^{-\corON{{\left(x-y\right)^{2}}/{2L}}}$.
\end{remark}
%{\bf  Proof of \corON{Lemma} \ref{lem: Brownian Barrier-1}: }
\begin{proof}[Proof of \corON{Lemma} \ref{lem: Brownian Barrier-1}]

It follows from Lemma 2.6, Lemma 2.7 and Proposition 6.1 of \cite{Bramson1ConvergenceofSolutionsOfKolmogorovEqn}  
that for any $C>0$ there is an $r$ large enough so that for  all $L>2r$
\[
\begin{array}{l}
P_{x}^{W}\left(f_{a, b}\left(l;L\right)+Cl_{L}^{\frac{1}{2}-\varepsilon}\le W_{l}\le f_{x, y'}\left(l;L\right)+Cl_{L}^{\frac{1}{2}+\varepsilon},l\in\left[r,L-r\right]|W_{L}=y'\right)\\
\ge cP_{x}^{W}\left(f_{a, b}\left(l;L\right)-Cl_{L}^{\frac{1}{2}-\varepsilon}\le W_{l},l\in\left[r,L-r\right]|W_{L}=y'\right).
\end{array}
\]

Let 
\[
A=\left\{ f_{a, b}\left(l;L\right)+Cl_{L}^{\frac{1}{2}-\varepsilon}\le W_{l}\le f_{x, y'}\left(l;L\right)+Cl_{L}^{\frac{1}{2}+\varepsilon},l\in\left[1,r\right]\cup\left[L-r,L-1\right]\right\} .
\]
For any fixed $r>0$ and and for all $l\ge1$ we have
\[
  \begin{array}{l}
    P_{x}^{W}\left(A|f_{a, b}\left(l;L\right)+Cl_{L}^{\frac{1}{2}-\varepsilon}\le W_{l}\le f_{x, y'}\left(l;L\right)+Cl_{L}^{\frac{1}{2}+\varepsilon},l\in\left[r,L-r\right],W_{L}=y'\right)\\
   \;\;\;\;\; 
%   \;\;\;\;\; 
%   \;\;\;\;\; 
%   \;\;\;\;\; 
%   \;\;\;\;\; 
%   \;\;\;\;\; 
%   \;\;\;\;\; 
%   \;\;\;\;\; 
%   \;\;\;\;\; 
   \ge c>0,
  \end{array}
\]
for a constant $c$ depending only on $r$ and $C$. Therefore
also
\[
\begin{array}{l}
P_{x}^{W}\left(f_{a, b}\left(l;L\right)+Cl_{L}^{\frac{1}{2}-\varepsilon}\le W_{l}\le f_{x, y'}\left(l;L\right)+Cl_{L}^{\frac{1}{2}+\varepsilon},l\in\left[1,L-1\right]|W_{L}=y'\right)\\
\ge cP_{x}^{W}\left(f_{a, b}\left(l;L\right)-Cl_{L}^{\frac{1}{2}-\varepsilon}\le W_{l},l\in\left[1,L-1\right]|W_{L}=y'\right).
\end{array}
\]

  Since the inequality in the opposite direction, with $c=1$, is trivial, we see by  sandwiching that we now  have shown
\begin{eqnarray}
&&P_{x}^{W}\left(f_{a, b}\left(l;L\right)+Cl_{L}^{\frac{1}{2}-\varepsilon}\le W_{l}\le f_{x, y'}\left(l;L\right)+Cl_{L}^{\frac{1}{2}+\varepsilon},l\in\left[1,L-1\right]|W_{L}=y'\right)
\nonumber\\
&&\asymp   P_{x}^{W}\left(f_{a, b}\left(l;L\right) \le W_{l},l\in\left[1,L-1\right]|W_{L}=y'\right)\nonumber\\
&&\asymp   P_{x}^{W}\left(f_{a, b}\left(l;L\right)-Cl_{L}^{\frac{1}{2}-\varepsilon}\le W_{l},l\in\left[1,L-1\right]|W_{L}=y'\right).\nonumber
\end{eqnarray}
The claim \eqref{eq: Brownian Barrier} then follows by first  using (\ref{eq: Brownian Bridge Barrier}), then multiplying by \corDB{the Gaussian density}, and then integrating over 
%$[y,y+\de]$ 
\corDB{$y' \in H_{y,\delta}$},
using the fact that \corDB{$f_{x, y'}\left(l;L\right) \le f_{x, y+\delta}\left(l;L\right) \leq f_{x, y}\left(l;L\right)+\de$}.
\end{proof}

The next lemma shows that we can replace the barrier $f_{a, b}\left(l;L\right)-Cl_{L}^{\frac{1}{2}-\varepsilon}$
``checked'' for all $l\in\left[1,L-1\right]$ with one   ``checked'' only
at integer times $l=1,\ldots,L-1$.
\bl
\label{lem: Brownian Barrier discrete}For all fixed $C>0,\varepsilon\in\left(0,\frac{1}{2}\right),\delta>0$ \corON{and $\eta>1$},
one has uniformly in $a\le x,b\le y,$  
$\left|x-y\right|\le 
\corON{\eta L}$ and $L$ large enough,
\begin{equation}
\begin{array}{l}
  P_{x}^{W}\left(f_{a, b}\left(l;L\right)-Cl_{L}^{\frac{1}{2}-\varepsilon}\le W_{l},l=1,\ldots,L-1,W_{L}\in\corON{H_{y,\delta}}\right)\\
  %\left[y,y+\delta\right]\right)\\
\;\;\;\;
\le c\frac{\left(1+x-a\right)\left(1+y-b\right)}{L}\frac{1}{\sqrt{L}}e^{-\frac{\left(x-y\right)^{2}}{2L}}.
\end{array}\label{eq: Brownian Barrier discrete}
\end{equation}
\el
\begin{proof}
Let  
\[
A_{L}=\left\{ f_{a, b}\left(l;L\right)-Cl_{L}^{\frac{1}{2}-\varepsilon}\le W_{l},l=1,\ldots,L-1\right\} ,
\]
and for a $\theta>1$ to be fixed later
\bea
&&
h_{l}\left(u,v\right)=P_{u}^{W}\left(f_{a, b}\left(l+t;L\right)-\theta C\left(l+t\right)_{L}^{\frac{1}{2}-\varepsilon}\le W_{t},t\in\left[0,1\right]|W_{1}=v\right).\nn
\eea
Note that  $h_{l}\left(u,v\right)$ is monotone increasing in $u,v$.
Hence using the Markov property we  have 
\begin{equation}
\begin{array}{l}
P_{x}^{W}\left(f_{a, b}\left(l\right)-\theta Cl_{L}^{\frac{1}{2}-\varepsilon}\le W_{l},l\in\left[1,L-1\right],W_{L}\in
\corON{H_{y,\delta}}\right)\\
%\left[y,y+\delta\right]\right)\\
=\corON{E}_{x}^{W}\left(\prod_{l=1}^{L-2}h_{l}\left(W_{l},W_{l+1}\right);W_{L}\in
\corON{H_{y,\delta}}\right)\\
%\left[y,y+\delta\right]\right)\\
\ge \corON{E}_{x}^{W}\left(\prod_{l=1}^{L-2}h_{l}\left(W_{l},W_{l+1}\right);A_{L};W_{L}\in
\corON{H_{y,\delta}}\right)\\
%\left[y,y+\delta\right]\right)\\
\ge\prod_{l=1}^{L-2}h_{l}\left(f_{a, b}\left(l\right)-Cl_{L}^{\frac{1}{2}-\varepsilon},f_{a, b}\left(l+1\right)-C\left(l+1\right)_{L}^{\frac{1}{2}-\varepsilon}\right)P_{x}^{W}\left(A_{L};W_{L}\in
%\left[y,y+\delta\right]\right).
\corON{H_{y,\delta}}\right).
\end{array}\label{eq: cont to disc barrier}
\end{equation}
But using (\ref{eq: Brownian bridge barrier exact})  with the $L$ there equal to $1$, we have   
\[
\begin{array}{l}
\sum_{l=1}^{L-2}\left(1-h_{l}\left(f_{a, b}\left(l\right)-Cl_{L}^{\frac{1}{2}-\varepsilon},f_{a, b}\left(l+1\right)-C\left(l+1\right)_{L}^{\frac{1}{2}-\varepsilon}\right)\right)\\
\le\sum_{l=1}^{L-2}\exp\left(-c(\theta-1)^{2}C^{2}l_{L}^{\frac{1}{2}-\varepsilon}\right)\le\frac{1}{2},
\end{array}
\]
if we let $\theta$ be large enough depending on $C$, but independently
of $L$. Thus (\ref{eq: Brownian Barrier discrete}) follows from (\ref{eq: Brownian Barrier}), Remark \ref{rem-ub}
and (\ref{eq: cont to disc barrier}).
\end{proof}
\corOF{ \begin{remark}
\label{rem-ub-2.3}
As in Remark 
\ref{rem-ub},
the condition $\left|x-y\right|\le
\corON{\eta L}$
%{L}/{\delta}$ 
can be dropped, at the cost of replacing the right hand side of 
(\ref{eq: Brownian Barrier discrete}) by
\eqref{expcomp}.
\end{remark}}
We now derive the corresponding estimates for   the 0-dimensional Bessel process by
applying a change of measure. \corDB{To formally introduce the Bessel process recall that the squared Bessel process $Y^{2}_{t}$ is a Markov process, $\mbox{BESQ}^{0}$, with semigroup, for $x>0$,
\begin{equation}
V_{t}(x,y)=e^{-x/2}\de_{0}(dy)+ 1_{\left(0,\infty\right)}(y){1 \over 2t}\(\frac{x}{y}\)^{1/2}\,\,   I_{1}(\sqrt{xy}/t)e^{-(x+y)/2t}\,dy,\label{besqref}
\end{equation}
see \cite[Chapter IX, Corollary (1.4)]{RY}, 
where 
\begin{equation}
I_{1}(z)=\sum_{k=0}^{\ff} {(z/2)^{2k+1} \over k!(k+1)!},\label{besseries}
\end{equation}
see \cite[8.447.2]{GR}, is the first Bessel function. Let $P^Y_y$  denote the law of the $Y_t$ starting at $Y_0 = z$. From \eqref{besqref} the equality \eqref{eq: Bessel Proc Marginal} easily follows.
%of imaginary argument.
}

\bl
\label{lem: Bessel Barrier}
For all fixed $\delta>0,\varepsilon\in\left(0,\frac{1}{2}\right)$ 
%$\tilde{C}\geq C+\de>0$ 
\corON{and $\eta>1$}, there exists 
\corOF{$c=c(\delta,\eta,\varepsilon,C )$ so that} for
%we have uniformly in 
$1\le x,y$, $0\le a\le x$, $0\le b\le y$,
and $\left|x-y\right|\le
\corON{\eta L}$, 
%%%$\left(x-a\right)\left(y-b\right)\le
%%%\corON{\eta L}$, 
%%%$\left|x-y\right|\le
%%%\corON{\eta L}$, 
%%%$\max (ab, |a-b|)\geq \corON{L/\eta}$,
%and $L$ large
%enough
\begin{eqnarray}
\label{eq: Bessel Proc Asymp UB}
&&P_{x}^{Y}\left(f_{a, b}\left(l;L\right)-Cl_{L}^{\frac{1}{2}-\varepsilon}\le Y_{l},l=1,\ldots,L-1,Y_{L}\in
\corON{H_{y,\delta}}\right)\nonumber \\
&&\leq c 
\frac{\left(1+x-a\right)\left(1+y-b\right)}{L}\sqrt{\corON{
  \frac{x}{yL}}}e^{-\frac{\left(x-y\right)^{2}}{2L}}.
\end{eqnarray}
If in addition  we assume that
$\left(x-a\right)\left(y-b\right)\le
\corON{\eta L}$, 
and $\max (ab, |a-b|)\geq \corON{L/\eta}$,
then for any
$\tilde{C}\geq C+\de>0$ 
there exists 
\corOF{$c'=c'(\delta,\eta,\varepsilon,C,\tilde C)$ so that} for
\begin{equation}
\label{eq: Bessel Proc Asymp LB}
\begin{array}{l}
P_{x}^{Y}\left(f_{a, b}\left(l;L\right)+Cl_{L}^{\frac{1}{2}-\varepsilon}\le Y_{l}\le f_{x, y}\left(l;L\right)+\tilde{C}l_{L}^{\frac{1}{2}+\varepsilon},\right.\\
\hspace{2.1in}l=1,\ldots,L-1,Y_{L}\in
\corON{H_{y,\delta}}\Big)\\
\geq c' \frac{\left(1+x-a\right)\left(1+y-b\right)}{L}\sqrt{\corON{
  \frac{x}{yL}}}e^{-\frac{\left(x-y\right)^{2}}{2L}}.
\end{array}
%\left[y,y+\delta\right]\right)\\
\end{equation}
%
%
%\begin{equation}
%\begin{array}{l}
%P_{x}^{Y}\left(f_{a, b}\left(l;L\right)+Cl_{L}^{\frac{1}{2}-\varepsilon}\le Y_{l}\le f_{x, y}\left(l;L\right)+\tilde{C}l_{L}^{\frac{1}{2}+\varepsilon},\right.\\
%\left. \hspace{2.1in}l=1,\ldots,L-1,Y_{L}\in
%\corON{H_{y,\delta}}\right)\\
%%\left[y,y+\delta\right]\right)\\
%\asymp P_{x}^{Y}\left(f_{a, b}\left(l;L\right)-Cl_{L}^{\frac{1}{2}-\varepsilon}\le Y_{l},l=1,\ldots,L-1,Y_{L}\in
%\corON{H_{y,\delta}}\right)\\
%%\left[y,y+\delta\right]\right)\\
%\asymp\frac{\left(1+x-a\right)\left(1+y-b\right)}{L}\sqrt{\corON{
%  \frac{x}{yL}}}e^{-\frac{\left(x-y\right)^{2}}{2L}}.
%\end{array}\label{eq: Bessel Proc Asymp}
%\end{equation}
\el
\corOF{ \begin{remark}
\label{rem-ub-2.4}
\corDBb{a)}
As in Remark 
\ref{rem-ub},
the condition $\left|x-y\right|\le
\corON{\eta L}$
%{L}/{\delta}$ 
can be dropped in
\eqref{eq: Bessel Proc Asymp UB},
at the cost of replacing the right hand side 
by
\[\frac{\left(1+x-a\right)\left(1+y-b\right)}{L}\sqrt{\corON{
  \frac{x}{yL}}}\sup_{z\in H_{y,\delta}}
e^{-\frac{\left(x-z\right)^{2}}{2L}}.\]
\corDBb{b) This is where we need the condition $\max (ab, |a-b|)\geq\corON{ L/\eta}$ from Theorem \ref{thm: GW Barrier} b).
It guarantees} that 
\begin{equation}
\int_{1}^{ L-1}\frac{1}{f^{ 2}_{a, b}\left(s;L\right)}\,ds=\frac{L-2}{ab+\frac{(   a-b)^{ 2}}{L}\left(1-\frac{1}{L}\right)}\leq c( \corON{  \eta})\label{integralUB}
\end{equation}
independent of $L\geq 3$. 
\corOF{The bound \eqref{integralUB}
	is used \corDBb{below} to bound the Radon--Nikodym derivative
	of a 0-Bessel process with respect to Brownian motion.}
%see the proof of Lemma \ref{lem: Bessel Barrier}.
\end{remark}}

%\begin{remark}The conditions $\left(x-a\right)\left(y-b\right)\le
%  \corON{\eta L}$ and $\max (ab, |a-b|)\geq  \corON{L/\eta}$ 
%  are not needed for the upper bound.
%\end{remark}
\begin{proof}[Proof of Lemma \ref{lem: Bessel Barrier}]
Recall that $Y_{t}$ solves the SDE $dY_{t}=dW_{t}-\frac{1}{2Y_{t}}dt$
for a Brownian motion $W_{t}$, 
\corON{until $\tau_0$, the time it hits $0$} \corDB{(not to be confused with the hitting time
	$\tau_x$ of a vertex $x$ of the binary tree, a notation used elsewhere in this paper)}. By  Girsanov's theorem  
\corON{(applied until $Y_t$ hits $\epsilon$) and monotone convergence}
it follows that for any $F\in \FF_{t}$,
\be
\corON{E}_{x}^{Y} \left(1_{\{   \tau_{0}>t\}}\,F\right)=\corON{E}_{x}^{W}
\left(1_{\{   \tau_{0}>t\}}\,F\,\,\sqrt{\frac{x}{W_{t}}}\exp\left(-\frac{3}{8}\int_{0}^{t}\frac{1}{W_{s}^{2}}ds\right)\right).\label{Girs.1}
\ee
We also note that $0$ is an absorbing boundary for the $0$-dimensional Bessel process, so that 
$Y_{L}>0$ implies that $\tau_{0}>L$. Therefore, since $y\geq 1$,
\begin{eqnarray*}
  \corON{\cal I}_{1}&:=&P_{x}^{Y}\left(f_{a, b}\left(l;L\right)-Cl_{L}^{\frac{1}{2}-\varepsilon}\le Y_{l},l=1,\ldots,L-1,Y_{L}\in
\corON{H_{y,\delta}}\right)\\
  %\left[y,y+\delta\right]\right)
%\label{}\\
&=&P_{x}^{Y}\left(f_{a, b}\left(l;L\right)-Cl_{L}^{\frac{1}{2}-\varepsilon}\le Y_{l},l=1,\ldots,L-1,Y_{L}\in
%\left[y,y+\delta\right]
\corON{H_{y,\delta}}, \corON{
  {   \tau_{0}>L}}\right)\nn\\
&\le&\sqrt{\frac{x}{y}}P_{x}^{W}\left(f_{a, b}\left(l;L\right)-Cl_{L}^{\frac{1}{2}-\varepsilon}\le W_{l},l=1,\ldots,L-1,Y_{L}\in
%\left[y,y+\delta\right]\right).\nn
\corON{H_{y,\delta}}\right),\nn
\end{eqnarray*}
by  (\ref{Girs.1}). Then using (\ref{eq: Brownian Barrier discrete}),
\begin{equation}
  \corON{\cal I}_{1}\leq c\sqrt{\frac{x}{y}}\frac{\left(1+x-a\right)\left(1+y-b\right)}{L}\frac{1}{\sqrt{L}}e^{-\frac{\left(x-y\right)^{2}}{2L}}.\label{eq: Bessel Proc UB}
\end{equation}
\corOF{This completes the proof of \eqref{eq: Bessel Proc Asymp UB}.}

\corOF{We turn to the proof of the lower bound 
\eqref{eq: Bessel Proc Asymp LB}.
Letting}  
\[
\wt A=\left\{ f_{a, b}\left(l;L\right)+Cl_{L}^{\frac{1}{2}-\varepsilon}\le W_{l}\le f_{x, y}\left(l;L\right)+\tilde{C}l_{L}^{\frac{1}{2}+\varepsilon},l\in\left[1,L-1\right]\right\} ,
\]
we have, by (\ref{Girs.1}),
\begin{equation}
\begin{array}{l}
P_{x}^{Y}\left(\wt A,Y_{L}\in
\corON{H_{y,\delta}}\right)
%\left[y,y+\delta\right]\right)\\
\ge\sqrt{\frac{x}{y+\delta}}\corON{E}_{x}^{W}\left(
\exp\left(-\frac{3}{8}\int_{0}^{L}\frac{1}{W_{t}^{2}}dt\right);
\wt A;W_{L}\in
\corON{H_{y,\delta}}\right)\\
%\left[y,y+\delta\right];
%\exp\left(-\frac{3}{8}\int_{0}^{L}\frac{1}{W_{t}^{2}}dt\right)\right)\\
\ge c\sqrt{\frac{x}{y}}
\corON{E}_{x}^{W}\left(\exp\left(-\frac{3}{8}\int_{0}^{1}\frac{1}{W_{t}^{2}}dt
\right)\cdot \exp\left(-\frac{3}{8}\int_{L-1}^{L}\frac{1}{W_{t}^{2}}dt\right);
\wt A; W_{L}\in
\corON{H_{y,\delta}}\right),
%\left[y,y+\delta\right]\right)
\end{array}\label{eq: Bessel Prob LB}
\end{equation}
where the last inequality follows because on the event $W_{l}\ge f_{a, b}\left(l;L\right)+Cl_{L}^{\frac{1}{2}-\varepsilon},l\in\left[1,L-1\right]$,
we have $\int_{1}^{L-1}\frac{1}{W_{t}^{2}}dt\le \corON{c(   \eta)} $ 
by (\ref{integralUB}).
%$  \int_{1}^{L-1}\frac{1}{f_{a, b}\left(t;L\right)^{2}}dt=\frac{L-2}{\left(a+\frac{b-a}{L}\right)\left(b+\frac{a-b}%{L}\right)}\le c$
%since $L\ge3$ and $\left(a\vee1\right)\left(b\vee1\right)\ge\delta L$.  

Furthermore for any $u,v\ge \corON{1/\eta}$ we have
%\delta\wedge1$ we have
\bea
&&
J_{u,v}:=\corON{E}_{u}^{W}\left(\exp\left(-\frac{1}{2}\int_{0}^{1}
\frac{1}{W_{t}^{2}}dt\right)\Big |W_{1}=v\right)\nn\\
&&
=\corON{E}_{0}^{W}\left(\exp\left(-\frac{1}{2}\int_{0}^{1}\frac{1}{(   W_{t}+u+t( v  -u))^{2}}dt\right)\Big |W_{1}=0\right)\geq c>0\nn,
\eea
where the constant depends only on \corON{$\eta$}, 
since the standard Brownian bridge  has positive probability 
of staying between 
$\corON{1/2\eta}$
%$(   \delta\wedge1)/2$ 
and 
$-\corON{1/2\eta}$.
%$-(   \delta\wedge1)/2$. 
We apply this to (\ref{eq: Bessel Prob LB}), where, after the 
Markov property we need to lower bound $J_{x,W_{1}}$ and $J_{W_{L-1},v}$ for 
$v\in\left[y,y+\delta\right]$.  Since $x,y\geq 1$ and
\corDB{$\max (ab, |a-b|)\geq L/\eta$ implies that}
\corON{$f_{a, b}\left(1;L\right),f_{a, b}\left(L-1;L\right)\ge
1/\eta$},
we therefore have in fact 
\begin{equation}
\begin{array}{l}
P_{x}^{Y}\left(f_{a, b}\left(l;L\right)+Cl_{L}^{\frac{1}{2}-\varepsilon}\le Y_{l}\le f_{x, y}\left(l;L\right)+\tilde{C}l_{L}^{\frac{1}{2}+\varepsilon},l\in\left[1,L-1\right],Y_{L}\in
\corON{H_{y,\delta}}\right)\\
%\left[y,y+\delta\right]\right)\\
\;\;\;\;
\ge c\sqrt{\frac{x}{y}}P_{x}^{W}\left(\wt A,;W_{L}\in
\corON{H_{y,\delta}}\right)
%\left[y,y+\delta\right]\right)\\
\ge c\sqrt{\frac{x}{y}}\frac{\left(1+x-a\right)\left(1+y-b\right)}{L}\frac{1}{\sqrt{L}}e^{-\frac{\left(x-y\right)^{2}}{2L}},
\end{array}\label{eq: Bessel Proc LB}
\end{equation}
using (\ref{eq: Brownian Barrier}). 
\corOF{This yields \eqref{eq: Bessel Proc Asymp UB}.}
%The two bounds (\ref{eq: Bessel Proc UB})
%and (\ref{eq: Bessel Proc LB}) together imply (\ref{eq: Bessel Proc Asymp}).
\end{proof}

\section{Local times and traversal counts}
\label{sec-3}

Let $X_{t},t\ge0,$ be the continuous time random
walk on the weighted graph $\left\{ 0,1,2,\ldots\right\} $ with unit
weights on each edge.  That is, the continuous time random
walk on $\left\{ 0,1,2,\ldots\right\} $ with exponential holding times of mean $1$ at $0$ and mean $1/2$ at all other points, and whose  jump chain  is simple random walk reflected at the origin. Let $L^{x}_{t}$ denote the local time at $x$ and $\tau (s)=\inf \{t>0\,|\,L^{0}_{t}=s\}$, the inverse local time at $0$. Fix $u>0$ once and for all. We use $\mathbb{P}_{u}$ to denote the probability for the process $\bar X_{t}=X_{t\wedge \tau (u)},t\ge 0$.  Let $\mathcal{L}_{l}=L^{l}_{\tau (u)}$, the total local time of $\bar X$ at $l$, and let $T_{l},l\ge0,$ be the
discrete traversal count $l\to l+1$. Under $\mathbb{P}_{u}$ the $T_{l},l\ge0,$
and $\mathcal{L}_{l},l\ge0,$ together have a Markovian structure,
as   stated by the following lemma.  

\bl[\corON{Markovian structure}]$\mbox{ }$
\label{lem: coupling}
%\el
\begin{itemize} 
\item[a)] The sequence $\left(Z_{n}\right)_{n\ge0}$ defined by 
\[
\left(Z_{0},Z_{1},Z_{2},\ldots\right)=\left(\mathcal{L}_{0},T_{0},\mathcal{L}_{1},T_{1},\mathcal{L}_{2},T_{2},\ldots\right)
\]
is a  time inhomogeneous Markov chain under $\mathbb{P}_{u}$. When $n=2k$ is even the law of $Z_{n+1}=T_{k}$
conditioned on $Z_{n}=\mathcal{L}_{k}=v$ is Poisson with parameter
$v$, and when $n=2k+1$ is odd the law of   $Z_{n+1}=\mathcal{L}_{k+1}$
conditioned on $Z_{n}=T_{k}=m$ is $0$ if $m=0$ and is the gamma distribution with
shape $m$ and scale parameter $1$ if $m\geq 1$.

\item[b)] Under $\mathbb{P}_{u}\left(\cdot|T_{1}=m_{1},\ldots,T_{L}=m_{L}\right)$
the $\mathcal{L}_{1},\ldots,\mathcal{L}_{L+1}$ are independent, $\mathcal{L}_{l}$
depends only on $T_{l-1}$ and $T_{l}$, and the $\mathbb{P}_{u}\left(\cdot|T_{l-1}=m_{l-1},T_{l}=m_{l}\right)$-law
of $\mathcal{L}_{l}$ is the gamma distribution with shape parameter
$m_{l-1}+m_{l}$ and scale parameter $\frac{1}{2}$.

\item[c)] Under $\mathbb{P}_{u}\left(\cdot|\mathcal{L}_{1}=y_{1},\ldots,\mathcal{L}_{L+1}=y_{L+1}\right)$
the $T_{1},\ldots,T_{L}$ are independent, $T_{l}$ depends only on
$\mathcal{L}_{l}$ and $\mathcal{L}_{l+1}$, and  with $u_{l}, u_{l+1}>0$,
\begin{equation}
\mathbb{P}_{u}\left(T_{l}=m|\mathcal{L}_{l}=u_{l},\mathcal{L}_{l+1}=u_{l+1}\right)=\frac{\left(u_{l}u_{l+1}\right)^{m}/\left(\left(m-1\right)!m!\right)}{\sqrt{u_{l}u_{l+1}}I_{1}\left(2\sqrt{u_{l}u_{l+1}}\right)}.\label{eq: cond traversal bessel dist}
\end{equation}

\item[d)] $\left(T_{l}\right)_{l\ge0}$ under $\mathbb{P}_{u}$ is a Markov
chain with the same transition kernel as that of $\left(T_{l}\right)_{l\ge0}$
under $P_{\cdot}$.

\item[e)] $\left(\mathcal{L}_{l}\right)_{l\ge0}$ under $\mathbb{P}_{u}$ is ${1 \over 2}\mbox{ BESQ}^{0}(2u)$ at integer times.   That is, $\left(\mathcal{L}_{l}\right)_{l\ge0}$ under $\mathbb{P}_{u}$ is
a Markov chain with the same transition kernel as $\left(\frac{1}{2}Y_{l}^{2}\right)_{l\ge0}$
under  $P_{\cdot}^{Y}$. 
\end{itemize}
\el 

\begin{remark} 
The distribution in (\ref{eq: cond traversal bessel dist}) is a generalization to $\nu=-1$ of the Bessel distribution $\mbox{{\it Bessel}}(\nu, z)$ defined in  \cite{PY2} for $\nu>-1$. This distribution appears in \cite{Feller}. See also  \cite{YK}.
\end{remark}
\begin{proof}[Proof of Lemma \ref{lem: coupling}]
Construct a collection of independent standard Poisson processes,
one for each directed edge. The random walk is constructed by placing
a local time clock at each vertex which only advances when the walker
is at that vertex. When \corDB{the walker is at a given vertex}
its local time clock advances until one of the two independent Poisson processes
associated with edges originating at the vertex register an
arrival, and then \corOF{the walker} 
\corDB{traverses} the corresponding edge. Note that the minimum of two independent exponentials of mean 1 is an exponential of mean $1/2$. From this construction
it is clear that given $u$ the count $T_{0}$ is the number of arrivals
up to time $u$ of the point process associated to  the edge  $0\to1,$ and therefore is
Poisson with parameter
$u$. Given $T_{0}=m$, then $\mathcal{L}_{1}$, the amount of local time spent
at vertex $1$ until the walker has returned to $0$ a total of $m$
times,  is the time until the $m$-th
arrival of the Poisson process associated to the edge $1\to0,$ and therefore
gamma distributed with
shape $m$ and scale parameter $1$. Iterating this a) follows and  d) is immediate. 

b) \corDB{It follows from the fact that $Z_{n}$ is a Markov chain that under\\ $\mathbb{P}_{u}\left(\cdot|T_{1}=m_{1},\ldots,T_{L}=m_{L}\right)$
the $\mathcal{L}_{1},\ldots, \mathcal{L}_{L+1}$ are independent and that the conditional law of
$\mathcal{L}_l$ depends only on $T_{l-1}$ and $T_l$. The characterization of the conditional law}
follows immediately since each local time $\mathcal{L}_{l}$ is the sum of independent mean $1/2$ exponentials, one for every visit to $l$, and $m_{l-1}+m_{l}$ is the number of such visits. 
 
c) \corDB{Similarly to above the conditional independence and dependence of the conditional law of $T_{l}$
only on $\mathcal{L}_l$ and $\mathcal{L}_{l+1}$ follows from the Markovian structure.}
From
the joint law, with $u_{l}>0$ and $m_{l}\geq 1$,
\bea
\mathbb{P}_{u}\left(T_{l}=m_{l},\mathcal{L}_{l+1}=du_{l+1}|\mathcal{L}_{l}=u_{l}\right)&=& e^{-u_{l}}\frac{u_{l}^{m_{l}}}{m_{l}!}\times\frac{u_{l+1}^{m_{l}-1}}{\left(m_{l}-1\right)!}e^{-u_{l+1}},\nn\\
&=& \frac{ 1}{u_{l+1}}\times \frac{(u_{l}u_{l+1})^{m_{l}}}{m_{l}!\left(m_{l}-1\right)!} e^{-(u_{l}+u_{l+1})},\nn
\eea
while
\begin{equation}
\mathbb{P}_{u}\left(T_{l}=0,\mathcal{L}_{l+1}=du_{l+1}|\mathcal{L}_{l}=u_{l}\right)=e^{-u_{l}}\de_{0}(du_{l+1}).\label{}
\end{equation}
Hence, using (\ref{besseries}),  we obtain
\bea
&&
\mathbb{P}_{u}\left(\mathcal{L}_{l+1}=du_{l+1}|\mathcal{L}_{l}=u_{l}\right)\label{besstrans}\\
&&= e^{-u_{l}}\de_{0}(du_{l+1})+\frac{ 1}{u_{l+1}} \sum_{m\ge 1}{ \left(u_{l}u_{l+1}\right)^{m}\over \left(m-1\right)!m!}\,\, e^{-(u_{l}+u_{l+1})}\nn\\
&&=e^{-u_{l}}\de_{0}(du_{l+1})+ \(\frac{u_{l}}{u_{l+1}}\)^{1/2}\,\,   I_{1}(2\sqrt{ u_{l}u_{l+1}})e^{-(u_{l}+u_{l+1})},\nn
\eea
and
one concludes that for $m\geq 1$, which is equivalent to $u_{l+1}>0$,
\[
\mathbb{P}_{u}\left(T_{l}=m|\mathcal{L}_{l}=u_{l},\mathcal{L}_{l+1}=u_{l+1}\right)=\frac{\left(u_{l}u_{l+1}\right)^{m}/\left(\left(m-1\right)!m!\right)}{\sqrt{u_{l}u_{l+1}}I_{1}\left(2\sqrt{u_{l}u_{l+1}}\right)}.
\]

e) This  follows by comparing (\ref{besstrans}) with (\ref{besqref}). 
\end{proof}

We will use  $\mathbb{Q}$ to denote the law of the time inhomogeneous Markov chain
$\left(Z_{n}\right)_{n\ge 0}$ under  $\mathbb{P}_{u} $. For $F\in \si(Z_{n}, Z_{n+1},\ldots)$     let
\begin{equation}
\mathbb{Q}_{n}^{x}=\mathbb{Q}\(F\,|\, Z_{n}=x\).\label{defq.1}
\end{equation}
Thus $\mathbb{Q}=\mathbb{Q}_{0}^{u}$. 

\corOF{For future reference we restate part e) of Lemma
\ref{lem: coupling} as}
%\ref{eq: Bessel Proc Asymp UB}the last part of the previous Lemma as 
\begin{equation}
\sqrt{2Z_{2l}},\hspace{.1 in} l=0,1, \ldots  \mbox{ under } \mathbb{Q} \mbox{ is a Bes}^{0}(\sqrt{2u})\Big|_{\mathbb{N}}.\label{defq.2}
\end{equation}
In the following we often replace the $Z$'s by the equivalent $\mathcal{L}$'s  and $T $'s.

 The following lemma gives some estimates
for one step transitions of the Markov chain.

\bl[\corON{One step estimates}]\mbox{}
\label{lem: One step estimates}
%\el
\begin{itemize}
\item[a)] For all $x>0$ such that $x^{2}/2$ is an integer and all $l\geq 1$,
\bea
&&
\mathbb{Q}_{2l-1}^{x^{2}/2}\(\left|\sqrt{2Z_{2l}}-x\right|\ge z\)\nn\\
&&=\mathbb{Q}\left(\left|\sqrt{2\mathcal{L}_{l}}-x\right|\ge z\,\Big|\,\sqrt{2T_{l-1}}=x \right)\le e^{-cz^{2}} \label{eq: count  to loc time transition bound}.
\eea
%{\bf   needs proof}

\item[b)] For all $-\sqrt{2}\le a<b$,
\begin{equation}
\inf_{\frac{x^{2}}{2}\in\left\{ 1,2,\ldots\right\} }\mathbb{Q} \left(\sqrt{2\mathcal{L}_{l}}-x\in\left[a,b\right]\,\Big|\,\sqrt{2T_{l-1}}=x \right)\ge c>0,\label{eq: one step loc time CLT}
\end{equation}
with a constant $c$ depending only on $a,b$.
\item[c)] For all $x>0$, and all $l\geq 1$,
\bea
&&
\mathbb{Q}_{2l}^{x^{2}/2}\(\left|\sqrt{2Z_{2l+1}}-x\right|\ge z\)\label{eq: loc time to  count transition bound}\\
&&
=\mathbb{Q} \left(\left|\sqrt{2T_{l}}-x\right|\ge z\,\Big|\,\sqrt{2\mathcal{L}_{l}}=x \right)\le e^{-cz^{2}}.\nn
\eea
%{\bf   needs proof}
\item[d)] For all $-1\le a<b$,
  %such that 
\begin{equation}
\inf_{x\ge1:x+\left[a,b\right]\cap\sqrt{2\mathbb{N}}\ne\emptyset}\mathbb{Q} \left(\sqrt{2T_{l}}-x\in\left[a,b\right]\,\Big|\,\sqrt{2\mathcal{L}_{l}}=x\right)\ge c>0,\label{eq: one set count clt}
\end{equation}
with a constant $c$ depending only on $a,b$.
\end{itemize}
\el
\begin{proof}[Proof of Lemma \ref{lem: One step estimates}]
Standard large deviation bounds for the gamma and Poisson distributions
give (\ref{eq: count  to loc time transition bound}) and   (\ref{eq: loc time to  count transition bound}). \corON{Indeed, to see 
\eqref{eq: loc time to  count transition bound}, 
recall that by Lemma \ref{lem: coupling}~a), the conditional
law of $T_l$, when conditioned on $\mathcal{L}_{l-1}=x^2/2$, is 
Poisson of parameter $x^2/2$. Since the Legendre transform of the
logarithmic moment generating function of the
Poisson distribution of parameter $\lambda$ is
\begin{equation}
  \label{eq-IP}
  I_{P,\lambda}(x)=\lambda-x+x\log(x/\lambda)\,,
\end{equation}
we obtain from Chebyshev's inequality that, for $z\geq 0$,
\begin{eqnarray}
  \label{eq-IP1}
&&  \log \mathbb{Q} \left(\sqrt{2T_{l}}-x\ge z\,\Big|\,
\sqrt{2\mathcal{L}_{l}}=x \right)\le
-I_{P,x^2/2}\left(\frac{(z+x)^2}{2}
\right) \nn \\
&&\;\;\;=-
\left[\frac{x^2}{2}+\frac{(x+z)^2}{2}\left(2\log\left(1+\frac{z}{x}
\right)-1\right)\right]\,,\end{eqnarray}
and 
\begin{eqnarray}
  \label{eq-IP2}
  &&\log \mathbb{Q} \left(\sqrt{2T_{l}}-x\le -z\,\Big|\,
\sqrt{2\mathcal{L}_{l}} =x \right)\le -I_{P,x^2/2}
\left(\frac{(x-z)^2}{2}
\right)\nn\\
&&\;\;\;=-
\left[\frac{x^2}{2}+\frac{(x-z)^2}{2}\left(2\log\left(1-\frac{z}{x}
\right)-1\right)\right]\,.
\end{eqnarray}
Separating according to whether $z/x>1$ (only relevant for \eqref{eq-IP1})
or $z/x\leq 1$, and performing some algebra yields 
\eqref{eq: loc time to  count transition bound}. The argument for 
(\ref{eq: count  to loc time transition bound})  is similar: 
recall, again from  Lemma~\ref{lem: coupling}~a), that the conditional
law of $\mathcal{L}_l$, when conditioned on $T_{l-1}=x^2/2$, is 
gamma of shape parameter $x^2/2$ and scale parameter  $1$. 
Recall that the Legendre transform of 
the logarithmic moment generating function 
of a Gamma variable of shape parameter $m$ and scale parameter $1$
equals 
\begin{equation}
  \label{eq-IG}
  I_{\Gamma,\lambda}(x)=x-m+m\log(m/x)\,.
\end{equation}
Replacing the expression in  \eqref{eq-IP} with \eqref{eq-IG} and
then using Chebyshev's inequality to obtain the analogue of 
\eqref{eq-IP1} and \eqref{eq-IP2} (with $m=x^2/2$),  the
proof then proceeds 
similarly to the Poisson case.}

To see (\ref{eq: one step loc time CLT}), note that by the central
limit theorem for a sum of independent exponentials,
\[
\lim_{x\to\infty,\frac{x^{2}}{2}\in\left\{ 1,2,\ldots\right\} }\mathbb{Q}_{ 1}^{x^{2}/2}\left(\mathcal{L}_{1}\in\frac{1}{2}\left[\left(x+a\right)^{2},\left(x+b\right)^{2}\right]\right)=\int_{a}^{b}\frac{1}{\sqrt{4\pi}}e^{-\frac{u^{2}}{4}}du>0.
\]
Thus for a large enough $\tilde{c}>0$ we have for $x^{2}/2\ge\tilde{c}$
that
\[
\mathbb{Q}_{ 1}^{x^{2}/2}\left(\sqrt{2\mathcal{L}_{1}}-x\in\left[a,b\right]\right)\ge c>0.
\]
Since the density $f\left(z\right)$ of a gamma random variable with
shape parameter $x^{2}/2$ and scale parameter $1$ satisfies $f\left(z\right)\ge c>0$
for all $0\le z\le\frac{1}{2}\left(\tilde{c}+b\right)^{2}$ if $1\le x^{2}/2\le\tilde{c}$,
it follows that the same holds for $1\le x^{2}/2\le\tilde{c}$. This
proves (\ref{eq: one step loc time CLT}). Similarily (\ref{eq: one set count clt})
follows from the central limit theorem for the Poisson distribution,
and since by our assumptions on $\left[a,b\right]$ imply that the
interval always contains an element from $\sqrt{2\mathbb{N}}$.
\end{proof}

 The following lemma gives some estimates
for conditioning the Markov chain at two times, before and after. 

\bl[\corON{Local time and transversal bounds}]\mbox{}
\label{lem: Loc Time and Traversals bounds}
\begin{itemize}
\item[a)] For all $m_{l-1},m_{l}$ and   $0\le z\le\sqrt{ m_{l-1}+m_{l} }$,
\begin{equation}
\mathbb{Q}_{ 1}^{m_{0}}\left(\sqrt{2\mathcal{L}_{l}}\ge\sqrt{ m_{l-1}+m_{l} }-z\,\Big|\,T_{l-1}=m_{l-1},T_{l}=m_{l}\right)\ge1-\exp\left(-z^{2}\right).\label{eq: cond loc time LD bound}
\end{equation}
\item[b)] For all $-\sqrt{2}\le a<b$,
\begin{equation}
  \inf_{\stackrel{t_{l-1}>0}{t_{l}\ge0}}\mathbb{Q}_{ 1}^{m_{0}}\left(\sqrt{2\mathcal{L}_{l}}-\sqrt{ m_{l-1}+m_{l} }\in\left[a,b\right]\,\Big|\,T_{l-1}=m_{l-1},T_{l}=m_{l}\right)\ge c,\label{eq: cond loc time clt bound}
\end{equation}
for all $l$, and a constant $c>0$ depending only on $a,b$.
\item[c)] For all $u_{l},u_{l+1}$ and $z\ge0$
\begin{equation}
\mathbb{Q}_{ 1}^{m_{0}}\left(|\sqrt{2T_{l}}-\left(4u_{l}u_{l+1}\right)^{1/4}|\ge z\,\Big|\,\mathcal{L}_{l}=u_{l},\mathcal{L}_{l+1}=u_{l+1}\right)\le
c\exp\left(-cz^{2}\right).\label{eq: excursion time upper bound from local time}
\end{equation}
\item[d)] \nc For all $  a<b$, with \[D_{a,b}=\{u_{l}>0,u_{l+1}\ge0\,|\,\left(u_{l}u_{l+1}\right)^{1/4}+\left[a,b\right]\cap \sqrt{2\mathbb{N} } \ne\emptyset\},\]
we have
\be 
\hspace{-.3 in}\inf_{u_{l},u_{l+1}\in D_{a,b}}\mathbb{Q}_{ 1}^{m_{0}}\left(\sqrt{2T_{l}}-\left(4u_{l}u_{l+1}\right)^{1/4}\in\left[a,b\right]\,\Big|\,\mathcal{L}_{l}=u_{l},\mathcal{L}_{l+1}=u_{l+1}\right)\ge c,\label{eq: cond exc clt bound}
\ee 
for some $c>0$ depending only on $a,b$.
\end{itemize}
\el
\begin{proof}
Similarly to (\ref{eq: one step loc time CLT}), the claim (\ref{eq: cond loc time clt bound})
follow from the central limit theorem for a sum of exponential random
variables, since under the conditioning $\mathcal{L}_{l}$  
have the gamma distribution with an integer shape parameter (see Lemma
\ref{lem: coupling} b). A large
deviation bound for the gamma distribution that can be proved using
the exponential Chebyshev inequality gives
\[
\mathbb{Q}_{ 1}^{m_{0}}\left(\mathcal{L}_{l}\le\frac{m_{l-1}+m_{l}}{2}-a\,\Big|\,T_{1}=m_{1},\ldots,T_{L}=m_{L}\right)\le\exp\left(-2\frac{a^{2}}{m_{l-1}+m_{l}}\right)
\]
for $0\le a\le\frac{m_{l-1}+m_{_{l}}}{2}. $
By taking $a=z\sqrt{\frac{m_{l-1}+m_{l}}{2}}$ this implies (\ref{eq: cond loc time LD bound}).

A basic large deviation estimate for the Bessel distribution (see
 \cite[Lemma 7.12]{BK} and its proof)  
\bea
&&
\mathbb{Q}_{ 1}^{m_{0}}\left(\left|T_{l}-\sqrt{u_{l}u_{l+1}}\right|\ge a\,\Big|\,\mathcal{L}_{l}=u_{l},\mathcal{L}_{l+1}=u_{l+1}\right)\nn\\
&&\hspace{.5 in}\le c\exp\left(-c\frac{a^{2}}{\sqrt{u_{l}u_{l+1}}}+c\frac{a}{\sqrt{u_{l}u_{l+1}}}\right)\text{ for all }a\ge0.
\eea
By taking $a=z\sqrt{u_{l}u_{l+1}}$ this implies (\ref{eq: excursion time upper bound from local time}).

\nc For \eqref{eq: cond exc clt bound}, set $z = \sqrt{u_l u_{l+1}}$ and $m = \frac{(\sqrt{2z}+v)^2}{2}$ in
\eqref{eq: cond traversal bessel dist} and use  the
estimate $I_{1}\left(z\right)\sim\frac{e^{z}}{\sqrt{2\pi z}}$
for $z\to\infty$, see \cite[8.451.5]{GR},
and Stirling's formula to see that for $v \in [a,b]$ and $z\geq z_{0}$ large enough
$$
\mathbb{Q}_{ 1}^{m_{0}}\left( \sqrt{2 T_l} =
\sqrt{2z} + v\,\Big|\,\mathcal{L}_{l}=u_{l},\mathcal{L}_{l+1}=u_{l+1}\right)
\ge
c \frac{e^{ -3v^2}}{\sqrt{z}}.
$$
 There are at least $c\sqrt{z}$ integers of the form $m = \frac{(\sqrt{2z}+v)^2}{2}$ for $v \in [a,b]$, so one obtains \eqref{eq: cond exc clt bound} for $z\geq z_{0}$. For smaller $z\leq z_{0}$ one simply uses the fact that the right-hand side of \eqref{eq: cond traversal bessel dist} is bounded away from zero for the finite number of  $m = \frac{(\sqrt{2z}+v)^2}{2}$ with $v \in [a,b]$ and $u_{l},u_{l+1}\in D_{a,b}   $.
\end{proof}

\section{\corDB{Galton-Watson process proofs}}
%\corON{Proof of Propositions \ref{prop: Almost LCLT} and \ref{thm: GW Barrier concise}}}
\label{sec-4}
\corDB{
We now have the necessary tools to prove the precise large deviation estimate Proposition \ref{prop: Almost LCLT}
and the barrier estimate Theorem \ref{thm: GW Barrier}
%Proposition \ref{thm: GW Barrier concise}
for the process $T_l$.} \corOF{ We start with the former.}
%We now prove the precise large deviation bound Proposition \ref{prop: Almost LCLT}.\\
\begin{proof}[Proof of Proposition \ref{prop: Almost LCLT}]
Since $T_{L}=0$ is the event that none of $x^{2}/2$ independent
excursions from $1$ to $0$ of a simple random walk on $\left\{ 0,1,2,\ldots\right\} $
hit $L+1$, and this has probability \corON{$1-1/(L+1)$}, the estimate
(\ref{eq: LCLT zero}) follows using the assumption $x\le \corON{\eta L}$.
%\frac{L}{\delta}$.

Turning to (\ref{eq: Almost LCLT}), we use that
\[
\mathbb{Q}_{1}^{x^{2}/2}\left(\sqrt{2T_{L}}\in
\corON{H_{y,\delta}}
%\left[y,y+\delta\right]
\right)\le
\frac{\mathbb{Q}_{1}^{x^{2}/2}\left(\sqrt{2\mathcal{L}_{L+1}}
\in \corON{H_{y,2\delta}}
%\left[y,y+2\delta\right]
\right)}{\mathbb{Q}_{1}^{x^{2}/2}\left(\sqrt{2\mathcal{L}_{L+1}}\in
\corON{H_{y,2\delta}}
%\left[y,y+2\delta\right]
|\sqrt{2T_{L}}\in \corON{H_{y,\delta}}
%\left[y,y+\delta\right]
\right)}.
\]
By (\ref{eq: one step loc time CLT}) we have that
\[
\mathbb{Q}_{1}^{x^{2}/2}\left(\sqrt{2\mathcal{L}_{L+1}}\in
\corON{H_{y,2\delta}}
%\left[y,y+2\delta\right]
|\sqrt{2T_{L}}\in \corON{H_{y,\delta}}
%\left[y,y+\delta\right]
\right)\ge c>0,
\]
for a constant $c$ depending only on $\delta$. Also
\[
\mathbb{Q}_{1}^{x^{2}/2}\left(\sqrt{2\mathcal{L}_{L}}\in
\corON{H_{y,\delta/2}}
%\left[y,y+\frac{\delta}{2}\right]
\right)\le\frac{\mathbb{Q}_{1}^{x^{2}/2}\left(\sqrt{2T_{L}}\in
\corON{H_{y,\delta}}
%\left[y,y+\delta\right]
\right)}{\mathbb{Q}_{1}^{x^{2}/2}\left(\sqrt{2T_{L}}\in
\corON{H_{y,\delta}}
%\left[y,y+\delta\right]
|\sqrt{2\mathcal{L}_{L}}\in\corON{H_{y,\delta/2}}
%\left[y,y+\frac{\delta}{2}\right]
\right)},
\]
and by (\ref{eq: one set count clt}) 
\[
\mathbb{Q}_{1}^{x^{2}/2}\left(\sqrt{2T_{L}}\in
\corON{H_{y,\delta}}
%\left[y,y+\delta\right]
|\sqrt{2\mathcal{L}_{L}}\in
\corON{H_{y,\delta/2}}
%\left[y,y+\frac{\delta}{2}\right]
\right)\ge c>0,
\]
for a constant $c$ depending only on $\delta$. Thus we have shown
that
\begin{equation}
\begin{array}{ccl}
c\mathbb{Q}_{1}^{x^{2}/2}\left(\sqrt{2\mathcal{L}_{L}}\in
\corON{H_{y,\delta/2}}
%\left[y,y+\frac{\delta}{2}\right]
\right) & \le & \mathbb{Q}_{1}^{x^{2}/2}\left(\sqrt{2T_{L}}\in
\corON{H_{y,\delta}}
%\left[y,y+\delta\right]
\right)\\
 & \le & \frac{1}{c}\mathbb{Q}_{1}^{x^{2}/2}
 \left(\sqrt{2\mathcal{L}_{L+1}}\in
 \corON{H_{y,2\delta}}
 %\left[y,y+2\delta\right]
 \right).
\end{array}\label{eq: sandwich}
\end{equation}
From Lemma \ref{lem: coupling} c) and (\ref{eq: Bessel Proc Marginal})
one obtains
\[
\begin{array}{lcl}
\mathbb{Q}_{1}^{x^{2}/2}\left(\sqrt{2\mathcal{L}_{L+1}}\in
\corON{H_{y,2\delta}}
%\left[y,y+2\delta\right]
\right) & = & \mathbb{Q}_{1}^{x^{2}/2}
\left(\mathbb{Q}_{2}^{\mathcal{L}_{1}}\left(\sqrt{2\mathcal{L}_{L}}\in
\corON{H_{y,2\delta}}
%\left[y,y+2\delta\right]
\right)\right)\\
 & = &\mathbb{Q}_{1}^{x^{2}/2}\left(P_{\sqrt{2\mathcal{L}_{1}}}^{Y}
 \left(Y_{L}\in\corON{H_{y,2\delta}}
 %\left[y,y+2\delta\right]
 \right)\right)\\
 & = & \mathbb{Q}_{1}^{x^{2}/2}\left(\int_{y}^{y+2\delta}
 \frac{\sqrt{2\mathcal{L}_{1}}}{L}e^{-\frac{2\mathcal{L}_{1}+
 z^{2}}{2L}}I_{1}\left(\frac{\sqrt{2\mathcal{L}_{1}}z}{L}\right)dz\right).
\end{array}
\]

\corON{The function}
$I_{1}\left(z\right)$, see (\ref{besseries}), is continuous, and clearly non-zero for $z>0$, $I_{1}\left(z\right)\sim z/2$ for $z\to 0$ and  $I_{1}\left(z\right)\sim\frac{e^{\corON{z}}}{\sqrt{2\pi z}}$
for $z\to\infty$, see \cite[8.451.5]{GR}.
In particular, $I_{1}\left(z\right)\leq C\frac{e^{z}}{\sqrt{z }}$. Using 
our assumptions on $x,y$ and (\ref{eq: count  to loc time transition bound}),
one sees that
\[
\mathbb{Q}_{1}^{x^{2}/2}\left(\sqrt{2\mathcal{L}_{L+1}}\in
\corON{H_{y,2\delta}}
%\left[y,y+2\delta\right]
\right)\le c\sqrt{\frac{x/y}{L}}e^{-\frac{x^{2}+y^{2}}{2L}}e^{\frac{xy}{L}}.
\]
Similarly, since  $I_{1}\left(z\right)\geq C'\frac{e^{z}}{\sqrt{z }}$ for $z\geq \de>0$, it follows from  (\ref{eq: count  to loc time transition bound}) that if  $\corON{L/\eta}\leq xy $
\[
\mathbb{Q}_{1}^{x^{2}/2}\left(\sqrt{2\mathcal{L}_{L}}\in
\corON{H_{y,\delta/2}}
%\left[y,y+\frac{\delta}{2}\right]
\right)\ge c'\sqrt{\frac{x/y}{L}}e^{-\frac{x^{2}+y^{2}}{2L}}e^{\frac{xy}{L}}.
\]
Together with (\ref{eq: sandwich}) this proves (\ref{eq: Almost LCLT}).
\end{proof}

\begin{proof}[Proof of Theorem \ref{thm: GW Barrier}]
%Proposition \ref{thm: GW Barrier concise}]
 \corOF{ We begin with the proof of the upper bound
 (\ref{eq: GW curved barrier upper bound})}.
%We first note that for that
% purpose,
% we may assume that
%%above also without the assumptions 
%%%$\left[y,y+\delta\right]
%$\corON{H_{y,\delta}}\cap\sqrt{2\mathbb{Z}^+}\ne\emptyset$,
%$\left(1+x-a\right)\left(1+y-b\right)\le\corON{\eta L}.$
%Indeeed, 
%without the first assumption the claim is trivially true 
%and without
%the second one can simply use Proposition \ref{prop: Almost LCLT}.  
%(If $y<\sqrt{2}$ we use the fact that
%$\{   \sqrt{2T_{L}}\in
%%%\left[y,y+\delta\right]
%%\corON{H_{y,\delta}}\}\subseteq \{   T_{L}=0\}\cup \{   
%  \sqrt{2T_{L}}\in
%  %\left[\sqrt{2},\sqrt{2}+\delta\right]
%  \corON{H_{\sqrt{2},\delta}}\}$.)}
%%{\color{red} WHAT ABOUT THE CONDITION  $\max (ab, |a-b|)\geq  L/\eta$ ????}
Fix $C>0$ and let 
\[
A=\left\{ f_{a, b}\left(l;L\right)-Cl_{L}^{\frac{1}{2}-\varepsilon}\le\sqrt{2T_{l}},l=1,\ldots,L-1,\sqrt{2T_{L}}\in
\corON{H_{y,\delta}}\right\}.
%\left[y,y+\delta\right]\right\}.
\]
We will first prove that
\begin{equation}
\mathbb{Q}_{1}^{x^{2}/2}\left(A\right)\le c\frac{\left(1+x-a\right)\left(1+y-b\right)}{L^{3/2}}\sqrt{\frac{x }{y}}\,e^{-\frac{\left(x-y\right)^{2}}{2L}} \corDB{\mbox{ for all }y\ge \sqrt{2}}. \label{eq: upper bound}
\end{equation}

Consider the event  
\bea
&&
B=\left\{ \sqrt{2\mathcal{L}_{1}}\ge \corDBa{a' + \frac{x-a}{2}},\sqrt{2\mathcal{L}_{l}}\ge f_{a'b}\left(l-1;L-1\right)-\bar{C}\left(l-1\right)_{L-1}^{\frac{1}{2}-\varepsilon},\right.\nn\\
&&\left. \hspace{1 in}\text{ for }l=2,\ldots,L-1;\sqrt{2\mathcal{L}_{L}}\ge\frac{y+b}{2}\right\}\nn
\eea
for a fixed $\bar{C}\ge1$ to be specified later, where
\be
a'=f_{a, b}\left(1;L\right),\text{ so that }f_{a'b}\left(l-1;L-1\right)=f_{a, b}\left(l;L\right).\label{a'def}
\ee
By Lemma \corON{\ref{lem: coupling}~b)},
\bea
&&
\mathbb{Q}_{1}^{x^{2}/2}\left(B\cap A\right)=\mathbb{Q}_{1}^{x^{2}/2}\lc 1_{A}\mathbb{Q}_{1}^{x^{2}/2}\left(\sqrt{2\mathcal{L}_{1}}\ge \corDBa{a' + \frac{x-a}{2}}\,\Big|\,T_{0},T_{1}\right)\right.\nn\\
&&
\times\left(\prod_{l=2}^{L-1}\mathbb{Q}_{1}^{x^{2}/2}\left(\sqrt{2\mathcal{L}_{l}}\ge f_{a'b}\left(l-1;L-1\right)-\bar{C}\left(l-1\right)_{L-1}^{\frac{1}{2}-\varepsilon}\,\Big|\,T_{l-1},T_{l}\right)\right)\nn\\
&&\left.\times\mathbb{Q}_{1}^{x^{2}/2}\left(\sqrt{2\mathcal{L}_{L}}\ge\frac{y+b}{2}\,\Big|\,T_{L-1},T_{L}\right)\rc.
\label{4.20}
\eea
On the event $A$ we have  
\[
\sqrt{T_{0}+T_{1} }\ge\frac{\sqrt{T_{0}}+\sqrt{T_{1}}}{\sqrt{2}}\ge \frac{x+a'-C}{2} \ge \corDB{a' + \frac{x-a}{2} - C - C'},
\]
\corDB{for a constant $C'$ depending only on $\eta$, where we used that $|a'-a| \le \eta $ by our assumptions on $x,y$.
Thus we have by (\ref{eq: cond loc time clt bound}) with say $\max(C+C',0)$ in place of $a$ and $\max(C+C',0)+1$ in place of $b$ that for some 
\corON{$c_{1}=c_{1}(C,\eta)>0$,}}
\be
1_{A}\mathbb{Q}_{1}^{x^{2}/2}\left(\sqrt{2\mathcal{L}_{1}}\ge a' + \frac{x-a}{2}\,\Big|\,T_{0},T_{1}\right)\ge 1_{A}c_{1}.\label{4.21}
\ee
Similarly for some $c_{2}>0$,
\be
1_{A}\mathbb{Q}_{1}^{x^{2}/2}\left(\sqrt{2\mathcal{L}_{L}}\ge \frac{y+b}{2}\,\Big|\,T_{L-1},T_{L}\right)\ge 1_{A}c_{2}.\label{4.22}
\ee
Note that for any $l'\leq L'$ with $L/2\leq L'\leq 2L$,
\be
f_{a, b}\left(l;L\right)=f_{a, b}\left(l';L'\right)+
(b-a)\!\left(\frac{l}{L}-\frac{l'}{L'}\right)\!=\!
f_{a, b}\!\left(l';L'\right)\!+
O_{\eta}\!\left(\left|l'-l\right|+\left|L-L'\right|\right),\label{4.22a}
\ee
and
\be
l_{L}^{\frac{1}{2}-\varepsilon}=\left(l'_{L'}\right)^{\frac{1}{2}-\varepsilon}+O_{\eta}\left(\left|l-l'\right|+\left|L-L'\right|\right),\label{4.22b}
\ee
which implies that on the event $A$,
\[
\sqrt{\frac{T_{l-1}+T_{l}}{2}}\ge\min\left(\sqrt{T_{l-1}},\sqrt{T_{l}}\right)\ge\frac{1}{\sqrt{2}}\left\{ f_{a, b}\left(l;L \right)-Cl_{L}^{\frac{1}{2}-\varepsilon}-c\right\} ,
\]
for $2\le l\le L-1$ for a constant $c$ depending only on \corON{$\eta$}
and $C$. Since $f_{a, b}\left(l;L \right)=f_{a', b}\left(l-1;L-1 \right)$ and $l_L,(l-1)_{L-1} \ge 1$ for $2 \le l \le L-1$ we have if we pick $\bar{C}$ large enough that on
$A$,
\[
\sqrt{\frac{T_{l-1}+T_{l}}{2}}\ge\frac{1}{\sqrt{2}}\left\{ f_{a'b}\left(l-1;L-1\right)-\frac{\bar{C}}{2}\left(l-1\right)_{L-1}^{\frac{1}{2}-\varepsilon}\right\} ,
\]
for $2\le l\le L-1$. Then by (\ref{eq: cond loc time LD bound})
we have
\bea 
&&
1_{A}\prod_{l=2}^{L-1}\mathbb{Q}_{1}^{x^{2}/2}\left(\sqrt{2\mathcal{L}_{l}}\ge f_{a'b}\left(l-1;L-1\right)-\bar{C}\left(l-1\right)_{L-1}^{\frac{1}{2}-\varepsilon}\,\Big|\,T_{l-1},T_{l}\right)\nn\\
&&
\ge1_{A}\prod_{l=2}^{L-1}\left(1-\exp\left(-\bar{C}^{2} \left(l-1\right)_{L-1}^{1-\varepsilon}\right)\right)\ge
\corON{\frac{1_A}{2}}, \label{4.23}
\eea
where the last inequality follows by making $\bar{C}$ large enough
(the choice can be made independent of $L$). Combining (\ref{4.20}), (\ref{4.21}), (\ref{4.22}) and (\ref{4.23}) we obtain $c\mathbb{Q}_{1}^{x^{2}/2}\left(A\right)\le\mathbb{Q}_{1}^{x^{2}/2}\left(A\cap B\right)$ with  $ c=
\corON{\frac{1}{2}c_1c_{2}}>0$,
so that 
\[
c\mathbb{Q}_{1}^{x^{2}/2}\left(A\right)\le\mathbb{Q}_{1}^{x^{2}/2}\left(A\cap B\right)\le\mathbb{Q}_{1}^{x^{2}/2}\left(B,\sqrt{2T_{L}}\in
\corON{H_{y,\delta}}\right).
%\left[y,y+\delta\right]\right).
\]
Recalling that $ T_{k}=Z_{2k+1}$ and $ \mathcal{L}_{k}=Z_{2k}$, the right-hand side can be written as   
\begin{eqnarray}
&&J_{1}:=\mathbb{Q}_{1}^{x^{2}/2}\left(\sqrt{2Z_{2}}\ge \corDB{a' + \frac{x-a}{2}},\sqrt{2Z_{2l}}\ge f_{a'b}\left(l-1;L-1\right)-\bar{C}\left(l-1\right)_{L-1}^{\frac{1}{2}-\varepsilon},\right.\nn\\
&&\left. \hspace{.5 in}\text{ for }l=2,\ldots,L-1;
\sqrt{2Z_{2L}}\ge\frac{y+b}{2}, \sqrt{2Z_{2L+1}}\in
\corON{H_{y,\delta}}\right).
%\left[y,y+\delta\right]\right)
\label{setup.1}
\end{eqnarray}
Using the Markov property 
\begin{eqnarray}
&&J_{1}=\mathbb{Q}_{1}^{x^{2}/2}\left(\sqrt{2Z_{2}}\ge\corDB{a' + \frac{x-a}{2}};\right.\label{setup.2}\\
&& \hspace{.5 in} \mathbb{Q}_{2 }^{Z_{2 }}\left(  \sqrt{2Z_{2l}}\ge f_{a'b}\left(l-1;L-1\right)-\bar{C}\left(l-1\right)_{L-1}^{\frac{1}{2}-\varepsilon}, \corON{l=2,\ldots,L-1};\right.\nn\\
&& 
%\left.\left.\text{ for }l=2,\ldots,L-1,
\hspace{1in}
\left.\left.\sqrt{2Z_{2L}}\ge\frac{y+b}{2};\mathbb{Q}_{2L}^{Z_{2L}}
\left( \sqrt{2Z_{2L+1}}\in
\corON{H_{y,\delta}}\right)
%\left[y,y+\delta\right]\right)
\right)\right).\nn
\end{eqnarray}
By 
(\ref{eq: count  to loc time transition bound}),
for all \corDBa{$j\geq a' + \frac{x-a}{2}$,} 
\begin{equation}
 \mathbb{Q}_{1 }^{x^{2}/2}\left( \sqrt{2Z_{2}}\in H_{j}\right)\leq e^{-c\left(j-x\right)^{2} }, \label{setup.4}
\end{equation}
and by  (\ref{eq: loc time  to  count  transition bound}),
for all $k\geq \frac{y+b}{2}$,  
\begin{equation}
\sup_{v\in H_{k}}\mathbb{Q}_{2L}^{v^{2}/2}\left( \sqrt{2Z_{2L+1}}\in
\corON{H_{y,\delta}}
%\left[y,y+\delta\right]
\right)\leq e^{-c\left(k-y\right)^{2} }. \label{setup.3}
\end{equation}
This shows that
\begin{eqnarray}
&&\hspace{-.2 in}J_{1}\leq   \sum_{j\ge\corDB{a' + \frac{x-a}{2}},
k\ge\frac{y+b}{2}}\sup_{z\in H_{j}}\mathbb{Q}_{2 }^{z^{2}/2}\left( 
\sqrt{2Z_{2l}}\ge f_{a'b}\left(l-1;L-1\right)-
\bar{C}\left(l-1\right)_{L-1}^{\frac{1}{2}-\varepsilon},\right.
\nn\\
&& \left.\hspace{.7 in} \text{ for }l=2,\ldots,L-1;\sqrt{2Z_{2L}}
\in H_{k}\right)e^{-c\left(j-x\right)^{2}}e^{-c\left(k-y\right)^{2} }.  \label{setup.5}
\end{eqnarray}
By (\ref{defq.2})
then
\begin{eqnarray}
&&K_{1}:=\mathbb{Q}_{2 }^{z^{2}/2}\left(  \sqrt{2Z_{2l}}\ge f_{a'b}\left(l-1;L-1\right)-\bar{C}\left(l-1\right)_{L-1}^{\frac{1}{2}-\varepsilon},\right.
\nn\\
&& \left.\hspace{1 in} \text{ for }l=2,\ldots,L-1;\sqrt{2Z_{2L}}\in H_{k}\right)
\label{setup.6}\\
&& =  P^Y_z\left( Y_{l-1}\ge f_{a'b}\left(l-1;L-1\right)-\bar{C}\left(l-1\right)_{L-1}^{\frac{1}{2}-\varepsilon},\right.
\nn\\
&& \left.\hspace{1 in} \text{ for }l=2,\ldots,L-1;Y_{L-1}\in H_{k} \right),\nonumber
\end{eqnarray}
where $Y_{t}$ is a $\mbox{Bes}^{0}(\sqrt{2u})$ process.
Thus using (\ref{eq: Bessel Proc Asymp UB})  we obtain that if $j,k\geq 1$ and  $|j-k|\leq 2\corON{\eta L}$
%L/\de$
\begin{eqnarray}
&&K_{1}=P^Y_z\left( Y_l\ge f_{a'b}\left(l;L-1\right)-
\bar{C} l_{L-1}^{\frac{1}{2}-\varepsilon},\right.
\nn\\
&& \left.\hspace{1 in} \text{ for }l=1,\ldots,L-2;
Y_{L-1}\in H_{k}\right)
\label{setup.7}\\
&& \leq C   \frac{\left(1+z-a'\right)\left(1+k-b\right)}{\left(L-1\right)^{3/2}}\sqrt{\frac{z }{k}}\,\,e^{-\frac{\left(z-k\right)^{2}}{2\left(L-1\right)}}\nonumber
\end{eqnarray}
Bounding $K_{1}$ by $1$ when $|j-k|> 2\corON{\eta L}$,
%L/\de$,  
it then follows  that $J_{1}$ is bounded by
\corON{
\bea
&&
\!\!
c\sum_{\stackrel{j\ge
  \corDB{a' + \frac{x-a}{2}}
}{k\ge\frac{y+b}{2}, 
|j-k|\leq 2
\eta L
%L/\de}
}}\frac{\left(1+j-a'\right)\left(1+k-b\right)}{\left(L-1\right)^{3/2}}\sqrt{\frac{j}{k}}
%\nn\\
% &&\hspace{3 in}\times 
e^{-\frac{\left(j-k\right)^{2}}{2\left(L-1\right)}}e^{-c\left(j-x\right)^{2}}e^{-c\left(k-y\right)^{2} }\nn\\
&&+\sum_{\stackrel{j\ge
  %\frac{x+a'}{2}
  \corOF{a' + \frac{x-a}{2}}
}{k\ge\frac{y+b}{2},
|j-k|> 2\eta L}}e^{-c\left(j-x\right)^{2}}e^{-c\left(k-y\right)^{2} }.\label{4.26}
\eea
}
The first sum is bounded by
\begin{equation}
c\frac{\left(1+\corDB{\frac{x-a}{2}}\right)\left(1+\corDB{\frac{y-b}{2}}\right)}{L^{3/2}} \sqrt{\frac{x}{y}}\,e^{-\frac{\left(x-y\right)^{2}}{2L}}.\label{}
\end{equation}
As for the second sum, our assumption that $x, y\leq \corON{\eta L}$ together with $|j-k|> 2 \corON{\eta L}$  implies that either $|x-j|\geq  \corON{\eta L}
$ or $|y-k|\geq  \corON{\eta L}$, (or both), in which case the second sum in (\ref{4.26}) is bounded by
$ce^{ -c'L^{ 2}}$.
\corOF{From this \eqref{eq: upper bound} follows.}
% for $y\geq \sqrt{2}$,
%\begin{equation}
%J_{1}\leq %c\frac{\left(1+x-a\right)\left(1+y-b\right)}{L^{3/2}}\sqrt{\frac{x}{y}}\,e^{-\frac{\left(x-y\right)^{2}}{2L}}.\label{}
%\end{equation}

The case $y=0$ requires special care. Note that when $y=b=0$ then
$\corON{L/\eta} \le a\le x\le
 \corON{\eta L}$ so $f_{a, b}\left(l;L\right)-Cl_{L}^{\frac{1}{2}-\varepsilon}>1$
 for all $l\le L-K$, for a constant $K$ depending on $C$ and $\corON{\eta}$.
Also for a large enough constant $C'$, we have for all $1\le k\le K$
that  
\[
f_{a0}\left(l;L\right)-Cl_{L}^{\frac{1}{2}-\varepsilon}\ge f_{a0}\left(l;L-k\right)-C'l_{L-k}^{\frac{1}{2}-\varepsilon}\text{\,\,\,for }l=1,\ldots,L-k-1.
\]
Thus letting
\[
A_{k}=\left\{\sqrt{2T_{l}}\geq  f_{a, b}\left(l;L-k\right)-C'l_{L-k}^{\frac{1}{2}-\varepsilon}\text{\, \,\,for }l=1,\ldots,L-k-1\right\} ,
\]
we have
\[
  \mathbb{Q}_{1}^{x^{2}/2}\left(A\right)\le\sum_{k=1}^{K}\mathbb{Q}_{1}^{x^{2}/2}\left(A_{k};\sqrt{2T_{L-k}}>1,T_{L-k+1}\in H_{0,\de\vee1} \right).
\]
Using the Markov property we have 
\[
%\begin{array}{l}
\mathbb{Q}_{1}^{x^{2}/2}\left(A_{k};\sqrt{2T_{L-k}}>1,T_{L-k+1}\in 
H_{0,\de\vee1}\right)
%\\
\le\sum_{j=1}^{ \ff}\mathbb{Q}_{1}^{x^{2}/2}\left(A_{k};\sqrt{2T_{L-k}}\in
\corON{H_j}
%\left[j,j+1\right]
\right)e^{-\frac{j^{2}}{2}}.
%\end{array}
\]
Now applying (\ref{eq: upper bound}) with $L-k$ in place of $L$ and using 
$ \corON{L/\eta}\le a\le x$
gives that this sum is at most
\[
\begin{array}{l}
\sum_{j=1}^{ \ff}c\frac{\left(1+x-a\right)\left(1+j\right)}{\left(L-k\right)^{3/2}} \sqrt{\frac{x}{j}}\,    e^{-\frac{\left(x-j\right)^{2}}{2\left(L-k\right)}}e^{-\frac{j^{2}}{2}}\\
\hspace{1cm}
\le c\frac{\left(1+x-a\right)}{\left(L-k\right) }e^{-\frac{x^{2}}{2\left(L-k\right)}}\le c\frac{\left(1+x-a\right)}{L }e^{-\frac{x^{2}}{2L}},
\end{array}
\]
for $1\le k\le K$, so the desired upper bound follows also in the
case $y=0$. In general, if $y\leq \sqrt{2}$ then $\left[y,y+\delta\right]\subseteq \left[0,\sqrt{2}+\delta\right]$ so the upper bound follows from the case $y=0$.
%We have thus shown that
%\begin{equation}
%\mathbb{Q}_{1}^{x^{2}/2}\left(A\right)\le c\frac{\left(1+x-a\right)}{L^{3/2}}\sqrt{x }\,e^{-\frac{\left(x-y\right)^{2}}{2L}} %\corDB{\mbox{ when } y=0}. \label{eq: upper bound endproof}
%\end{equation} 
This completes the proof of the upper bound (\ref{eq: GW curved barrier upper bound}).

We now  prove the lower bound. Recall that $x\geq \sqrt{2}$ and we first assume that also $y\geq \sqrt{2}$.
Let \nc
\[\mbox{{\it Tube}}_{C,\wt C }(l,L)=\{z\,|\,f_{a, b}\left(l;L\right)+Cl_{L}^{\frac{1}{2}-\varepsilon}\leq z\leq f_{x, y}\left(l;L\right)+\wt C l_{L}^{\frac{1}{2}+\varepsilon} \}\]
and set
\[
  D=\left\{ \sqrt{2T_{l}}\in \mbox{{\it Tube}}_{C,\wt C}(l,L),l=1,\ldots,L-1,\sqrt{2T_{L}}\in \corON{H_{y,\delta}}
%\left[y ,y+\de\right] 
\right\}.
\]
We claim that
\begin{equation}
\mbox{{\it Tube}}_{C,\wt C}(l,L)\cap\sqrt{2\mathbb{Z}^{+}}\ne\emptyset,\,\, \mbox{for all }l=1,\ldots,L-1, \label{pass.1}
\end{equation}
so that $D\ne\emptyset$. By the concavity of the square root, it suffices to show that $\wt C l_{L}^{\frac{1}{2}+\varepsilon}-Cl_{L}^{\frac{1}{2}-\varepsilon}\geq \sqrt{2}$ for all $l=1,\ldots,L-1$. 
This is certainly true for $l=1$ or $l=L-1$ by our requirements  that 
\corOF{imply that}
$\wt C  \geq C+\sqrt{2}$, and the general case 
follows since $f\(   z\)= \wt C z^{\frac{1}{2}+\varepsilon}-Cz^{\frac{1}{2}-\varepsilon}$ is monotone increasing in $z\geq 1$.

We now  show   that for $x,y\geq \sqrt{2}$  
\begin{equation}
\mathbb{Q}_{1}^{x^{2}/2}\left(D\right)\ge c\frac{\left(x-a\right)\left(y-b\right)}{L}\sqrt{\frac{x/y}{L}}\,e^{-\frac{\left(x-y\right)^{2}}{2L}}.\label{eq: lower bound}
\end{equation}

To this end consider the event  
\bea
&&
E =\left\{ \sqrt{2\mathcal{L}_{1}}\in H_{x}, \sqrt{2\mathcal{L}_{l}}\in \mbox{{\it Tube}}_{C',  C''}(l,L),l=2,\ldots,L-1; \right.\nn\\
&&\left. \hspace{2.5 in}\sqrt{2\mathcal{L}_{L}}\in\left[y+\frac{\de}{3},y+\frac{2\de}{3}\right] \right\},\nn
\eea
where $C'=C+\de/3$ and $  C''= C+\de $, so that $ C''-  C'\geq 2\de/3.$

We show below that for some   $\wh c>0$
\be
\mathbb{Q}_{1}^{x^{2}/2}\left(D \cap E \right)\ge \wh c\,\,\mathbb{Q}_{1}^{x^{2}/2}\left(E\right),\label{lbound.12}
\ee
from which it will follow that
\be
\mathbb{Q}_{1}^{x^{2}/2}\left(D \right)\ge\wh c\,\,\mathbb{Q}_{1}^{x^{2}/2}\left(E\right).\label{lbound.12z}
\ee
As before, using the Markov property
\begin{eqnarray}
&&\mathbb{Q}_{1}^{x^{2}/2}\left(E\right)=\mathbb{Q}_{1}^{x^{2}/2} \left( \sqrt{2\mathcal{L}_{1}}  \in H_{x}, \mathbb{Q}_{2}^{\mathcal{L}_{1}} \left(\sqrt{2\mathcal{L}_{l}}\in \mbox{{\it Tube}}_{C',C''}(l,L),\right.\right.\nn\\
&&\left.\left. \hspace{1.5 in}l=2,\ldots,L-1; \sqrt{2\mathcal{L}_{L}}\in\left[y+\frac{\de}{3},y+\frac{2\de}{3}\right] \right)\right).\nn
\label{setup.20}\\
&& \geq 
\mathbb{Q}_{1}^{x^{2}/2} \left( \sqrt{2\mathcal{L}_{1}}  \in H_{x}\right)
\inf_{z\in H_{x}} 
\mathbb{Q}_{2}^{z^{2}/2} \left(\sqrt{2\mathcal{L}_{l}}\in \mbox{{\it Tube}}_{C',C''}(l,L),\right. \nn\\
&&\left.  \hspace{1.5 in}l=2,\ldots,L-1; \sqrt{2\mathcal{L}_{L}}\in\left[y+\frac{\de}{3},y+\frac{2\de}{3}\right] \right)\nn\\
&& =: 
\mathbb{Q}_{1}^{x^{2}/2} \left( \sqrt{2\mathcal{L}_{1}}  \in H_{x}\right)
\corOF{\inf_{z\in H_{x}} K_2(z)}\,. 
\end{eqnarray}
By (\ref{eq: one step loc time CLT}) we have that  $\mathbb{Q}_{1}^{x^{2}/2} \left( \sqrt{2\mathcal{L}_{1}}  \in H_{x}\right)\geq c>0$ uniformly in all possible $x$.  Once again, by (\ref{defq.2}), if $Y=Y_{t}$ is a $\mbox{Bes}^{0}(\sqrt{2u})$ then
\corOF{\begin{eqnarray}
%&&K_{2}:=\mathbb{Q}_{2}^{z^{2}/2} \left(\sqrt{2\mathcal{L}_{l}}\in \mbox{{\it Tube}}_{C',C''}(l,L),\right. \nn\\
%&&\left.  \hspace{1  in}l=2,\ldots,L-1; \sqrt{2\mathcal{L}_{L}}\in\left[y+\frac{\de}{3},y+\frac{2\de}{3}\right] \right)
%\label{setup.21}\\
K_2(z)& = & P\left( Y_{l}\in \mbox{{\it Tube}}_{C',C''}(l+1,L),\right.
\label{setup.21}
\\
&& \left.\hspace{1 in} l=1,\ldots,L-1,Y_{L}\in \left[y+\frac{\de}{3},y+\frac{2\de}{3}\right]\,\Big |\, Y_{0}=z\right),
\nn
\end{eqnarray}
}
and using the Markov property and then (\ref{eq: Bessel Proc Asymp UB}) 
\corOF{(with $y$
replaced by $y+\delta/3$ and with $\delta$ taken as $2\delta/3$),}
which requires $x,y\geq 1$ and $C''-  C'\geq 2\de/3$,  we obtain  
%\begin{eqnarray}
%&&K_{2}=P^{z}\left( Y_{l}\in \mbox{{\it Tube}}_{2C,\frac{1}{2}\wt C}(l+1,L),\right.
\corON{\be
\label{setup.22}
%&& \left.\hspace{1 in} \text{ for }l=1,\ldots,L-1,Y_{L-1}\in \left[y+\frac{1}{3},y+\frac{2}{3}\right]\right)
%\nn\\
%&& 
K_2 \geq     c\frac{\left(1+x -a\right)\left(1+y-b\right)}{L-1}\sqrt{\frac{x/y}{L-1}}\,e^{-\frac{\left(x-y\right)^{2}}{2(L-1)} },
\ee}
%\nonumber
%\end{eqnarray}
which because of our assumptions on $x,y,a,b$ is greater than or equal to the right-hand
side of (\ref{eq: lower bound}). Together with (\ref{lbound.12z}) this proves (\ref{eq: lower bound})  for $x,y\geq \sqrt{2}$.

Assume now that $y\leq \sqrt{2}$, which by our assumption requires that $x\geq 
\corON{L/\eta}$. The tube at $l=L-1$ contains some interval $\left[z,z+\bar\delta\right]$ with $z\geq \sqrt{2}$ such that  $\left[z,z+\bar\delta\right]\cap\sqrt{2\mathbb{Z}^{+}}\neq \emptyset$. We then apply (\ref{eq: lower bound}) for the interval of length $L-1$ with \corON{$\sqrt{2\mathcal{L}_{L-1}}\in\left[z,z+\bar\delta\right] $}. Since now $\sqrt{\frac{x/z}{L-1}}\asymp 1$, it is easy to see that our desired lower bound follows from the fact that there is some positive probability to go from $\left[z,z+\bar\delta\right]$ to some point in $\left[y,y+\delta\right]\cap\sqrt{2\mathbb{Z}^{+}}$ in a single step.

We now turn to the proof of (\ref{lbound.12}).
\corOF{By part c) of Lemma \ref{lem: coupling},} 
\bea
\mathbb{Q}_{1}^{x^{2}/2}\left(D \cap E \right)&=&\mathbb{Q}_{1}^{x^{2}/2}\lc1_{E }\mathbb{Q}_{1}^{x^{2}/2}\left(\sqrt{2T_{1}}\in \mbox{{\it Tube}}_{  C, \wt C}(1,L)\,\Big|\,\mathcal{L}_{1},\mathcal{L}_{2}\right)\right.\nn\\
&&\times\left(\prod_{l=2}^{L-1}\mathbb{Q}_{1}^{x^{2}/2}\left(\sqrt{2T_{l}}\in \mbox{{\it Tube}}_{  C, \wt C}(l,L)\,\Big|\,\mathcal{L}_{l },\mathcal{L}_{l+1}\right)\right)\nn\\
&&\left.\times\mathbb{Q}_{1}^{x^{2}/2}\left(\sqrt{2T_{L}}\in H_{y}\,\Big|\,\mathcal{L}_{L } \right)\rc.\label{4.30}
\eea

Recall from (\ref{eq: cond traversal bessel dist}) that  the law of $\sqrt{2T_{l}}$ depends only on the product  $\mathcal{L}_{l }\mathcal{L}_{l+1}$. Note that
\begin{equation}
|f_{x, y}\left(l+1;L\right)-f_{x, y}\left(l;L\right)|\leq \eta\label{}
\end{equation}
with a similar bound for $f_{a, b}\left(l;L\right)$.
It follows   that on the event $E $
\be
\left\{ f_{a, b}\left(l;L\right)+C'l_{L}^{\frac{1}{2}-\varepsilon}-c'\right\}\le \(4\mathcal{L}_{l }\mathcal{L}_{l+1}\)^{ 1/4} \le \left\{ f_{x, y}\left(l;L\right)+C'' l_{L}^{\frac{1}{2}+\varepsilon}+ c''\right\}, \label{4.34}
\ee
for $2\le l\le L-1$ with $c'=\eta+C'$ and $c''=\eta+C''$. Thus  we have that on
$E $
\be
\{ |\sqrt{2T_{l}}- \(4\mathcal{L}_{l }\mathcal{L}_{l+1}\)^{ 1/4}|\leq  \frac{\de}{10}l_{L}^{\frac{1}{2}-\varepsilon}\}\subseteq \{   \sqrt{2T_{l}}\in \mbox{{\it Tube}}_{  C, \wt C}(l,L)\}, \label{4.35}
\ee
for $k\le l\le L-k$ for some fixed $k$ independent of $L$. Then by (\ref{eq: excursion time upper bound from local time})
we have
\bea 
&&
1_{E }\prod_{l=k}^{L-k}\mathbb{Q}_{1}^{x^{2}/2}\left(\sqrt{2T_{l}}\in \mbox{{\it Tube}}_{  C, \wt C}(l,L)\,\Big|\,\mathcal{L}_{l },\mathcal{L}_{l+1}\right)\nn\\
&&
\ge1_{E }\prod_{l=k}^{L-k}\left(1-\exp\left(-c \left(\frac{\de}{10}\right)^{2}l _{L}^{1-2\varepsilon}\right)\right)\ge1_{E }c_{\ast}, \label{4.33}
\eea
for some $c_{\ast}>0$  independent of $L$. 

 Consider now  $2\leq l\leq k$ or $L-k\leq l\leq L-1$. 
By  (\ref{4.34}), on the event $E $ we have 
  \be
\{  \sqrt{2T_{l}}\in \(4\mathcal{L}_{l }\mathcal{L}_{l+1}\)^{ 1/4}+[r,s]\}\subseteq \{   \sqrt{2T_{l}}\in \mbox{{\it Tube}}_{  C, \wt C}(l,L)\}, \label{4.35f}
\ee
with\[  [r,s]=[\eta+C',\eta+C''+\sqrt{2}].  \]
 Hence  using (\ref{eq: cond exc clt bound}) we see that for any
 \be
1_{E }\mathbb{Q}_{1}^{x^{2}/2}\left(\sqrt{2T_{l}}\in \mbox{{\it Tube}}_{ C, \wt C}(l,L)\,\Big|\,\mathcal{L}_{l},\mathcal{L}_{l+1}\right)\ge 1_{E }c_{l},\label{4.31b}
\ee
for some $c_{l}>0$ which depends on \corON{$\eta$}, $\de$ and $C$.

 In addition, on the event $E $ we have 
\bea
  &&
\sqrt{a}\sqrt{  f_{a, b}\left(2;L\right)+C'   }\le \(4\mathcal{L}_{1}\mathcal{L}_{2}\)^{ 1/4}\nn\\
&&\hspace{2 in} \le \sqrt{x+1}\sqrt{  f_{x, y}\left(2;L\right)+2C''   }.\nn
\eea 
Thus, as in the last paragraph,  we have by (\ref{eq: cond exc clt bound}) that 
\be
1_{E }\mathbb{Q}_{1}^{x^{2}/2}\left(\sqrt{2T_{1}}\in \mbox{{\it Tube}}_{ C, \wt C}(1,L)\,\Big|\,\mathcal{L}_{1},\mathcal{L}_{2}\right)\ge 1_{E }c_{1},\label{4.31a}
\ee
for some $c_{1}>0$ which depends on \corON{$\eta$}, $\de$ and $C$.

Finally, on the event $E $ we have  
\[
\sqrt{2\mathcal{L}_{L}}\in\left[y+\frac{\de}{3},y+\frac{2\de}{3}\right].
\]
For all $y\geq y_{0}$ sufficiently large, for any $v\in \left[y+\frac{\de}{3},y+\frac{2\de}{3}\right]$, we will have that $v+[0,\de/3]\cap\sqrt{2\mathbb{N}}\neq \emptyset$.
Then we have by (\ref{eq: one set count clt}) that for some  
for some $c_{L}>0$
\be
1_{E }\mathbb{Q}_{1}^{x^{2}/2}\left(\sqrt{2T_{L}}\in H_{y,\de}\,\Big|\,\mathcal{L}_{L } \right)\ge 1_{E }c_{L}.\label{4.32}
\ee 
On the other hand, we know from Lemma \ref{lem: coupling} a) that $T_{L}$ conditioned on $\mathcal{L}_{L}=u$ is Poisson with parameter $u$, and since by our assumption $H_{y,\de}\cap\sqrt{2\mathbb{Z}^{+}}\neq \emptyset$, it follows that for some possibly smaller $c_{L}>0$, (\ref{4.32}) also holds if $y\leq y_{0}$.

This completes the proof of (\ref{lbound.12}) with $\wh c=c_{\ast}\prod_{l=1}^{k}c_{l}\prod_{j=L-k}^{L}c_{j}>0$, and hence of
(\ref{eq: lower bound}).
%Finally (\ref{eq: upper bound}) and (\ref{eq: lower bound}) together
%prove (\ref{eq: gw concise}).   
\end{proof}

\section{Tree cover time}
\label{sec-5}
\nc
\corDBa{In this section we prove Theorem \ref{thm: tree cover main} about the tightness of the cover time of the regular tree.}
\corO{We begin this section
	by introducing notation. Recall that \corOF{$\mathcal{T}_L$ 
	denotes a rooted tree, with root $\rho$ to 
	%tree rooted at a vertex \corDB{$\rho$} to 
	which one attaches a standard binary tree of depth $L$. 
      Recall also that }
	$\mathbb{P}$ is the probability of simple random walk on this tree starting from \corDB{$\rho$}. For a vertex $v\in \mathcal{T}_L$,
	we write $|v|$ for the tree distance of $v$ from the root \corDB{$\rho$}. We define 
	level $l$ of the tree $(l=-1, 0,1,\ldots,L)$ as those vertices at distance 
	$l+1$ from \corDB{$\rho$}, and write $\ll_l$ for the vertices of 
	$\mathcal{T}_L$ at level $l$. For $y\in \ll_L$, we let $[y]_l$ denote 
	the ancestor of $y$ at level $l$, i.e. the unique vertex in $\ll_l$ on the geodesic connecting $y$ and \corDB{$\rho$};  
	thus $\left[y\right]_{-1}$ is the root \corDB{$\rho$}.}

Let $D_{n}$ be the time of the $n$-th return to the root \corDB{$\rho$}. For each
\corO{$y\in \ll_{L}$}, define
\[
T_{l}^{y,n}=\text{\# of traversals of the edge }\left[y\right]_{l-1}\to\left[y\right]_{l}\text{ up to time }D_{n}.
\]
%where $\left[y\right]_{l}$ denotes the ancestor of $y$ in level
%$l$ of the tree, and we take $\left[y\right]_{-1}$ to denote root.
The count $T_{l}^{y,n}$ is the edge local time of the directed edge
$\left(\left[y\right]_{l-1},\left[y\right]_{l}\right)$ accumlated
during the first $n$ excursions from the root. Note that for 
\corO{
	%a leaf
	$y\in \ll_L$} we have 
\begin{equation}
\hit_{y}>D_{n}\text{ if and only if }T_{L}^{y,n}=0.\label{eq: hitting time traversal count connection}
\end{equation}

When we observe the random walk on the tree only when it is on the
path from $\rho$ to  a leaf \corO{$y\neq \rho$}, 
we obtain a one dimensional \corDB{simple} random
walk on \corDB{$\{-1, 0,1,\ldots,L\}$}. As such the following is simply
a restatement of the classical fact that the directed edge local times
of 1D simple random walk are a Markov process, and more precisely
a critical Galton--Watson with geometric offspring distribution (this
is the discrete equivalent of the \corDB{second} Ray-Knight theorem).

\bl\label{lem: disc ray-knight}
For each $n$ and each leaf $\corON{y\in\ll_L}$ the process $\left(T_{l}^{y,n}\right)_{l=0,\ldots,L}$ has \corDB{law $P_n$}, i.e. is a Galton--Watson process with offspring distribution given by the
geometric law with mean $1$.
\el

Also note that
\begin{equation}\label{eq: reduced complexity}
	\begin{array}{c}
		\mbox{if $y,z$ are two different leaves 
		  in $\corOF{\ll_L}$
		whose paths from root to leaf}\\
		\mbox{ overlap up to level $k$, then $T_{l}^{y,n}=T_{l}^{n,z}$ for $l\le k$,}
	\end{array}
\end{equation}
and conditionally on $T_{k}^{y,n}=T_{k}^{n,z}$ the
processes $\left(T_{l}^{y,n}\right)_{l\ge k},\left(T_{l}^{n,z}\right)_{l\ge k}$
are independent. \corDBa{Thus,} in fact, the $T_{l}^{y,n}$ are a collection of branching
Galton--Watson process, like a branching random walk where the random
walk is replaced by the Galton--Watson process. This analogy allows
us to use methods from branching random walks to prove Theorem \ref{thm: tree cover main}.

Let $\left(T_{l}\right)_{l\ge0}$ denote a generic Galton--Watson process
with this offspring distribution. A basic bound for the increments
of $T_{l}$ is given by the following
\begin{equation}
\corO{P_n\left(\left|\sqrt{2T_{l}}-\sqrt{2T_{0}}\right|\ge z
	\right)\le ce^{-\frac{z^{2}}{2l}},\quad z\geq 0}\label{eq: tail bound}
\end{equation}
which is easily proved by computing the exponential moments of $T_{l}$
using the representation of $T_{l}$ as the sum of $T_{0}$ random
variables, each being the product of a Bernoulli with parameter $\frac{1}{l}$
and a Geometric with parameter $\frac{1}{l}$, see \cite[Lemma 4.6]{BK}
\corDB{(for $n = O(l^2)$ and $z=O(l)$ it is a special case of the more precise bound \eqref{eq: Almost LCLT})}.

\corO{Throughout this section we set
	$$\kappa =\kappa_L= \sqrt{2\log2}-\frac{\log L}{\sqrt{2\log2}L}. $$
}
%\corDB{ Note that for $x$ fixed as $L \to \infty$,
%	\begin{equation}
%	e^{-\frac{(\kappa L + x)^2}{\kappa L}} \asymp  2^{-L}Le^{-\kappa x },
%	\end{equation}
%	a fact we will use several times.}
The main step in proving the upper bound \corDB{ \eqref{eq: upper tail bound} on the right tail} 
is the following proposition
which gives an upper bound in terms of excursion counts - later we
relate the excursion counts to the actual clock of the random walk
$X_{n}$ to obtain from this the upper bound of Theorem \ref{thm: tree cover main}.
\begin{proposition}
	\label{prop: Upper Bound cover} There is a constant $c$ such that
	for all $L\ge1$ and $x\ge1$ we have
	\begin{equation}
	\corDB{\mathbb{P}}
	\left(\min_{\corO{y\in \ll_L}}T_{L}^{y,n}=0\right)\le 
	\corO{cx e^{-\kappa 
			x}},\label{eq: right tail in terms of excursions}
	\end{equation}
	where
	\corO{  $\sqrt{2n}=
		%\sqrt{2\log2}L-\frac{\log L}{\sqrt{2\log2}}+x=
		\kappa L+x$.}
\end{proposition}
\begin{proof}
	%Let 
	%\[
	%\kappa=\sqrt{2\log2}-\frac{\log L}{\sqrt{2\log2L}}\text{ so that }\sqrt{2n}=\kappa L-x.
	%\]
	The probability in (\ref{eq: right tail in terms of excursions})
	is at most
	\corO{
		\begin{eqnarray}
		&&\underset{\text{(I)}}{\underbrace{
				%\mathbb{P}
				\corDB{\mathbb{P}}
				\left(\exists y\in\ll_L,l\leq L-1\text{ s.t. }\sqrt{2T_{l}^{y,n}}
				< \alpha\left(l\right)\right)}}\\
		&& \;+\underset{\text{(II)}}{\underbrace{
				%\mathbb{P}
				\corDB{\mathbb{P}}
				\left(\exists y\in \ll_L\text{ s.t. }\sqrt{2T_{l}^{y,n}}\ge \alpha\left(l\right),l=1,\ldots,L-1,T_{L}^{y,n}=0\right)},}\label{eq: upper bound split}
		\nonumber
		\end{eqnarray}
	}
	where we have split the event according to whether or not the processes
	$T_{l}^{y,n}$ stay above the barrier 
	\corO{
		\[
		\alpha\left(l\right)=\kappa\left(L-l\right)-l_L^{1/6},
		\]
	}
	(as typically they should).
	
	By a union bound over the levels of the tree we have,
	%\corO{with $y$ any of the leaves in $\ll_L$},
	%\corO{with $\{T_l\}$ a copy of the process $\{T_l^{y,n}\}$},
	\begin{equation}
	\corO{\text{(I)}\le\sum_{k=1}^{L-2}\mathbb{P}\left(\sqrt{2T_{l}^{y,n}} \ge \alpha\left(l\right),l=1,\ldots,k,\sqrt{2T_{k+1}^{y,n}}<\alpha\left(k+1\right)\right).}\label{eq: sum over levels}
	\end{equation}
	\corDB{
	By Lemma \ref{lem: disc ray-knight}, \eqref{eq: reduced complexity} and a union bound over the $2^k$ vertices in the $k$-th level of the tree the $k$-th summand is at most
	\begin{equation}
		2^{k}\corDB{P_n}\left(\sqrt{2T_{l}} \ge \alpha\left(l\right),l=1,\ldots,k,\sqrt{2T_{k+1}}<\alpha\left(k+1\right)\right).\label{eq: sum over levels 2}
	\end{equation}}
	We condition on the height of the process at time $k$ to obtain an
	upper bound of
	\begin{equation}
	\begin{array}{c}
	\sum_{j\ge0}{\corDB{P_n}}\left(\sqrt{2T_{l}}\ge \alpha\left(l\right),l=1,\ldots,k-1,\sqrt{2T_{k}}\in H_{\alpha\left(k\right)+j}\right)\\
	\quad\quad\quad\quad\quad\times
	\sup_{z\in H_{\alpha\left(k\right)+j}}\corDB{P_{\frac{z^2}{2}}}\left(\sqrt{2T_{1}}\le \alpha\left(k+1\right) \right),
	\end{array}\label{eq: cond on height}
	\end{equation}
	for \corO{the probability in the
		summand in (\ref{eq: sum over levels})}. \corO{By (\ref{eq: tail bound}),
		\begin{equation}
		\label{eq-I5.6}
		\sup_{z\in H_{\alpha\left(k\right)+j}}\corDB{P_{\frac{z^2}{2}}}\left(\sqrt{2T_{1}}\le \alpha\left(k+1\right) \right)\leq
		%the second probability in \eqref{eq: cond on height} is at most 
		ce^{-j^{2}/2}.
		\end{equation}
	}
	\corO{To bound the first probability in \eqref{eq: cond on height}, we
		use
		(\ref{eq: GW curved barrier upper bound}) 
		with modified parameters, as follows.
		First, note that for any $\corDBa{\gamma}<1$, 
		$$ l_L^{\corDBa{\gamma}}\leq k_L^{\corDBa{\gamma}}+l_k^{\corDBa{\gamma}}.$$
		Setting,
		for $l=0,\ldots,k$, 
		$\hat \alpha(l)= f_{a,b}(l,L)-l_k^{1/6}$ with $a=\kappa L-k_L^{1/6}$ and 
		$b=\kappa(L-k)-k_L^{1/6}$, 
		we thus conclude that
		$\hat \alpha(l)\leq \alpha(l)$ with $\alpha(k)=\hat \alpha(k)$. Therefore,
		\begin{eqnarray}
		\label{eq-ofer1}
		&&\corDB{P_n}\left(\sqrt{2T_{l}}\ge \alpha\left(l\right),l=1,\ldots,k-1,\sqrt{2T_{k}}\in H_{\alpha\left(k\right)+j}\right) \\
		&&\leq 
		\corDB{P_n}\left(\sqrt{2T_{l}}\ge \hat \alpha\left(l\right),l=1,\ldots,k-1,\sqrt{2T_{k}}\in H_{\hat \alpha\left(k\right)+j}\right).
		\nonumber \end{eqnarray}
	}
	\corO{We now apply (\ref{eq: GW curved barrier upper bound}) with $L$ 
		replaced by $k$, \corDB{the $x$ of (\ref{eq: GW curved barrier upper bound}) replaced by $\sqrt{2n}$, $y=\hat \alpha (k)+j$}, and these values of $a,b$, and obtain that the right hand side in 
		\eqref{eq-ofer1} is bounded by}
	\corO{
		$$ c \frac{(x+k_L^{1/6})(j+1)}{k} e^{-(\kappa k+x-j+k_L^{1/6})^2/2k}
		\leq c (x+k_L^{1/6})(j+1) 2^{-k}e^{\kappa j-\kappa x-\kappa k_L^{1/6}}.$$
	}
	%the first is at most $2^{-k}e^{-cx-2\left(k\wedge\left(L-k\right)\right)^{1/6}}e^{2y}$.
	\corO{
		Combining with \eqref{eq-I5.6} and 
		summing over $j$ we get that (\ref{eq: cond on height}) is
		at most $c2^{-k}(x+k_L^{1/6}) e^{-\kappa (x+k_L^{1/6})}$.}
	Thus from (\ref{eq: sum over levels}) 
	\corO{
		\begin{equation}
		\text{(I)}\le c\sum_{k=1}^{L-2}2^{k}\left(2^{-k}
		(x+k_L^{1/6})e^{-\kappa
			x-\kappa k_L^{1/6}}\right)=ce^{-\kappa
			x}\sum_{k=1}^{L}(x+k_L^{1/6})
		e^{-\kappa k_L^{1/6}}\le cx e^{-\kappa x}.\label{eq: barrier first part}
		\end{equation}
	}
	
	By Lemma \ref{lem: disc ray-knight} and a union bound over all leaves \corO{in $\ll_L$ }
	and (\ref{eq: GW curved barrier upper bound})
	we have 
	\corDB{
		\[
		\text{(II)}\le2^{L}P_n\left(\sqrt{2T_{l}}\ge \alpha\left(l\right),l=1,\ldots,L,T_{L}=0\right)\le 
		%2^{L}\left(c2^{-L}e^{-\kappa x}\left(1+x\right)\right)\le 
		cxe^{-\kappa x},
		\]
	}
	if $x\ge1$. Together with (\ref{eq: barrier first part}) 
	\corO{and \eqref{eq: upper bound split}}, this proves
	(\ref{eq: right tail in terms of excursions}).
\end{proof}

\corDB{To prove a lower bound in terms of excursion counts we will} consider the processes
\[
T_{l}^{y,k,m}=\begin{array}{c}
\#\text{ of traversals from }\left[y\right]_{l-1}\text{ to }\left[y\right]_{l}\text{ during}\\
\text{the first } m\text{ excursions from }\left[y\right]_{k}\text{ to }\left[y\right]_{k-1}.
\end{array}
\]
Like $\left(T_{l}^{y,n}\right)_{l\ge0}$, the process $\left(T_{l+k}^{y,k,m}\right)_{l\ge0}$
is a critical Galton--Watson process with geometric offspring, but
with initial population $m$ \corDB{i.e.
\begin{equation}\label{eq: subtree trav count law}
	\mbox{the law of } \left(T_{l+k}^{y,k,m}\right)_{l\ge0} \mbox{ under }\mathbb{P} \mbox{ is } P_m.
\end{equation}}
%Note also that if the latest common
%ancestor of two leaves $y\ne z$ in the tree is in or above level
%$k$, then $\left(T_{l+k}^{y,k,m}\right)_{l\ge0}$ and $\left(T_{l+k}^{z,k,m}\right)_{l\ge0}$
%are independent.
\corDB{Also
\begin{equation}\label{eq: subtree trav count indep}
	\begin{array}{c}
		\left(T_{l+k}^{y,k,m}\right)_{l\ge0} \mbox{ and } \left(T_{l+k}^{z,k,m}\right)_{l\ge0}
		\mbox{ are independent if the}\\
		\mbox{latest common ancestor of leaves $y \ne z$ in the tree is in or above level $k$}.
	\end{array}
\end{equation}}
Note further that
\corDB{
	\begin{equation}\label{eq: subtree trav count compat}
		\mbox{on the event } \corO{T_{k}^{y,n}=m \mbox{ we have } T_{k+l}^{y,n}=T_{k+l}^{y,k,m}} \mbox{ for } l=0,\ldots L-k.
	\end{equation}
}
	%on the event \corO{$T_{k}^{y,n}=m$ we have $T_{k+l}^{y,n}=T_{k+l}^{y,k,m}$} for $l=0,\ldots L-k$.

\corDB{The main step in proving the lower bound on the right tail \eqref{eq: upper tail lower bound} and the upper bound on the left tail \eqref{eq: lower tail bound} is the following result for the processes $T^{y,k,m}_l$.}
\begin{lemma}\label{lem: subtree lower bound}
	There is a constant $c$ such that for all $x \ge 0$ and $k \ge 1$ we have
	\begin{equation}
	\liminf_{L\to\infty}
	%\mathbb{P}
	\corDB{\mathbb{P}}\left( \inf_{y \in \ll_L} T_{L}^{y,k,m}=0 \right) \ge \frac{2^k (1+x) e^{- x \sqrt{2\log2}}}{2^k (1+x) e^{- \sqrt{2\log2}x } + c}  .\label{eq: to prove lower subtree}
	\end{equation}
	where 
	\[
	\sqrt{2m}=\bar{\kappa}\bar{L} + x, \mbox{ for }
	\bar{L}=L-k, \mbox{ and } \bar{\kappa}= \kappa_{\bar L}  = \sqrt{2\log2}-\frac{1}{\sqrt{2\log2}}\frac{\log\bar{L}}{\bar{L}}.
	\]
\end{lemma}
\begin{proof}
	Consider the event 
	\corO{
		\[
		I_{y}=\left\{ \bar{\kappa}\left(\bar{L}-l\right)+\left(l_{\bar L}\right)^{1/4}\le\sqrt{2T_{k+l}^{y,k,m}}\text{ for }l=1,\ldots,\bar{L}-1,T_{k+\bar{L}}^{y,k,m}=0\right\} ,
		\]}
	that $\left(T_{l+k}^{y,k,m}\right)_{l\ge0}$ stays above \corDB{the barrier $\bar{\kappa}\left(\bar{L}-l\right)+\left(l_{\bar L}\right)^{1/4}$}
	and ends up at zero, and the counting random variable
	\[
	\mathcal{N}_{k}=\sum_{y\corO{\in \ll_L}}1_{\corO{I_y}}.
	\]
	If $\mathcal{N}_{k}\ge1$ then $\min_{y\corO{\in \ll_L}
		%\text{ leaf }
	}T_{L}^{y,k,m}=0$ so to prove \eqref{eq: to prove lower subtree} it suffices to show that for large enough $L$ \nc
	\begin{equation}
	%\mathbb{P}
	\corDB{\mathbb{P}}\left(\mathcal{N}_{k}\ge1\right)
	\ge \frac{2^k (1+x) e^{-\bar{\kappa} x}}{2^k (1+x) e^{-\bar{\kappa}x}+c}.\label{eq: to prove lower bound count}
	\end{equation}
	Using (\ref{eq: GW curved barrier upper bound}),
	(\ref{eq: gw lower bound}) one has
	\[
	%\mathbb{E}
	\corDB{\mathbb{E}}\left(\mathcal{N}_{k}\right)\asymp2^{L}2^{-\bar{L}} (1+x) e^{-\bar{\kappa} x } = 2^{k} (1+x) e^{-\bar{\kappa}x},
	\]
	and by the Paley-Zygmund inequality
	\begin{equation}
	%\mathbb{P}
	\corDB{\mathbb{P}}\left(\mathcal{N}_{k}\ge1\right)\ge\frac{
		%\mathbb{E}
		\corDB{\mathbb{E}}\left(\mathcal{N}_{k}\right)^{2}}{
		\corDB{\mathbb{E}}
		%\mathbb{E}
		\left(\mathcal{N}_{k}^{2}\right)}.\label{eq: paley zygmund}
	\end{equation}
	To prove (\ref{eq: to prove lower bound count}) we estimate the second
	moment of $\mathcal{N}_{k}$. We have
	\[
	\begin{array}{ccl}
	%\mathbb{E}
	\corDB{\mathbb{E}}\left(\mathcal{N}_{k}^{2}\right) & = & \sum_{y,z}
	%\mathbb{P}
	\corDB{\mathbb{P}} \left(I_{y}\cap I_{z}\right)\\
	& = & {\displaystyle \sum_{y,z:\text{ branch early}}}
	%\mathbb{P}
	\corDB{\mathbb{P}}\left(I_{y}\right)
	%\mathbb{P}
	\corDB{\mathbb{P}}\left(I_{z}\right)+{\displaystyle \sum_{y,z:\text{ branch late}}}
	%\mathbb{P}
	\corDB{\mathbb{P}}\left(I_{y}\cap I_{z}\right),
	\end{array}
	\]
	where the first sum is over pairs of leaves whose most recent common
	ancestor lie in level $k$ of the tree or above, so that the events
	$I_{y}$ and $I_{z}$ are independent, and the second sum is over
	all other pairs of leaves. The first sum is at most $2^{2L}
	%\mathbb{P}
	\corDB{\mathbb{P}}\left(I_{y}\right)^2=
	%\mathbb{E}
	\corDB{\mathbb{E}}\left(\mathcal{N}_{k}\right)^{2}$.
	Thus (\ref{eq: paley zygmund}) implies (\ref{eq: to prove lower bound count})
	once we have shown that \corDB{for large enough $L$}
	\begin{equation}
	\sum_{y,z:\text{ branch late}}
	%\mathbb{P}
	\corDB{\mathbb{P}}\left(I_{y}\cap I_{z}\right) \le c2^{k} (1+x) e^{ -\bar{\kappa} x}.\label{eq: to prove non-early branching}
	\end{equation}
	
	%\textcolor{blue}{Updated up to here; Lots of changes needed below.}
	
	To bound this sum we sum over the possible level of the common ancestor
	of $y,z$:
	\begin{equation}
	\sum_{y,z:\text{ branch late}}
	\corDB{\mathbb{P}}\left(I_{y}\cap I_{z}\right)
	\le2^{k}\sum_{j=0}^{\bar{L}}2^{2\bar{L}-j}p_{j},\label{eq: sum over late branching}
	\end{equation}
	where $p_{j}=
	%\mathbb{P}
	\corDB{\mathbb{P}}\left(I_{y}\cap I_{z}\right)$ for $y,z$ whose
	common ancestor is in level $k+j$ of the tree and we have used that
	there are at most $2^{k}2^{2\bar{L}-j}$ \corO{such pairs.} Now
	%vertices in that level. Now
	by conditioning on the value of the processes at the point where they
	branch, we have
	\corO{
		\begin{equation}
		p_{j}\leq 
		\sum_{u\ge0}q_{j}\left(u\right)r_{j}\left(u\right)^{2},\label{eq: pj as sum}
		\end{equation}
	}
	\corDB{where we sum over positive integers $u$ and where}
	\corDB{
		\[
		q_{j}\left(u\right)=P_{m}\left(\bar{\kappa}\left(\bar{L}-l\right)
		%+l_{\bar L}^{1/4}
		\le\sqrt{2T_{l}},l=1,\ldots,j-1,\sqrt{2T_{j}}\in H_{m_u}
		% [\sqrt{2m_{u}}, \sqrt{2m_u}+1]
		\right),
		\]
		and
		\begin{eqnarray*}
			&&r_{j}\left(u\right)\\
			%\max_{z\in H_{m_u}}\\
			&&
			\!\!\!\!\!\!\!
			\!\!\!\!\!\!\!
			\!\!\!\!\!\!\!
		=	\max_{z\in H_{m_u}}
			P_{\frac{z^2}{2}}\left(\bar{\kappa}\left(\bar{L}-j-l\right)
			%+\left(l+j\right)_{\bar L}^{1/4}
			\le\sqrt{2T_{l}},l=1,\ldots,\bar{L}-j-1,T_{\bar{L}-j}=0\right),
		\end{eqnarray*}
	}
	with
	\[
	m_{u}=\bar{\kappa}\left(\corO{\bar L}-j\right)+j_{\bar L}^{
		\corO{1/4}}+u.
	\]
	%(which is an upper bound since it would be enough to sum over
	%$u=u_i$ so that
	%$\sqrt{2m_{u_{i+1}}}-\sqrt{2m_{u_i}}=1$).}
	By \corO{\eqref{eq: GW curved barrier upper bound}},
	%(\ref{eq: GW curved barrier upper bound})
	\corO{
		\[
		q_{j}\left(u\right)\le c\frac{1}{j}\corDB{\left(1+x\right)}
		\corO{\left(1+j_{\bar L}^{\corO{1/4}}+u\right)}e^{-\frac{\left(j\bar{\kappa} + \corDB{x} -
				j_{\bar L}^{\corO{1/4}}-u\right)^{2}}{2j}},
		\]
	}
	and
	\corO{
		\[
		r_{j}\left(u\right)\le c\frac{1}{\bar{L}-j} \left(1+j_{\bar L}^{\corO{1/4}}+u\right)e^{-\frac{\left(\left(\bar{L}-j\right)\bar{\kappa}+
				j_{\bar L}^{\corO{1/4}}+u\right)^{2}}{2\left(\bar{L}-j\right)}}.
		\]
	}
	Since
	\corO{
		\begin{eqnarray*}
			&&\frac{\left(j\bar{\kappa} + \corDB{x} - j_{ \bar L}^{1/4}-u\right)^{2}}{2j}+
			2\frac{\left(\left(\bar{L}-j\right)\bar{\kappa}+j_{\bar L}^{1/4}+u\right)^{2}}{2\left(\bar{L}-j\right)}\\
			&&\;\ge\frac{\bar{\kappa}^{2}}{2}\left(2\bar{L}-j\right) + \corDB{\bar{\kappa}x} +cj_{\bar L}^{\corO{1/4}}+cu,
		\end{eqnarray*}
	}
	and $e^{-\frac{\bar{\kappa}^{2}}{2}\left(2\bar{L}-j\right)}=\bar{L}^{\frac{2\bar{L}-j}{\bar{L}}}2^{-\left(2\bar{L}-j\right)} \le \bar{L}(\bar L-j) 2^{-2 (\bar L - j)}$
	we thus have from (\ref{eq: pj as sum}) that
	\corO{
		\[
		p_{j}\le c2^{-\left(2\bar{L}-j\right)} \bar{L}(\bar L-j) \frac{\corO{
				1
			}}{j\left(\bar{L}-j\right)^{2}} \corDB{\left( 1+x \right)} \left(1 + \corDB{x} + j_{\bar L}^{1/4}\right)^3 e^{ -\corDB{\bar{\kappa}x} - c j_{\bar L}^{1/4}.}
			\]
		}
		Thus the right-hand side of (\ref{eq: sum over late branching}) is
		at most
		\corO{
			\[
			2^{k}\sum_{j=1}^{\bar{L}}\frac{\bar L}{j\left(\bar{L}-j\right)}
			\corDB{\left( 1+x \right)} \left(1+j_{\bar L}^{1/4}\right)^3 e^{ -\corDB{\bar{\kappa}x} - cj_{\bar L}^{1/4}}.
			\]
		}
		Since the sum is bounded by a constant for all $\bar{L}$ we get that
		\[
		\sum_{y,z:\text{ branch late}}
		%\mathbb{P}
		\corDB{\mathbb{P}}\left(I_{y}\cap I_{z}\right)\le c2^{k} \corDB{(1+x)e^{-\bar{\kappa}x}}, 
		\]
		which completes the proof of (\ref{eq: to prove non-early branching}),
		and therefore also of (\ref{eq: to prove lower bound count}) and
		(\ref{eq: lower bound in terms of excursions}).
	\end{proof}
	\corDB{By simply setting $k=0$ and bounding $x e^{ - x\sqrt{2\log2} }$ by a constant we obtain a lower bound on the right tail in terms of excursions. This will later lead to \eqref{eq: upper tail lower bound}.
	\begin{corollary}
		For all $x \ge 1$,
		\begin{equation}
		\liminf_{L\to\infty}
		%\mathbb{P}
		\corDB{\mathbb{P}}\left( \inf_{y \in \ll_L} T_{L}^{y,n}=0 \right) \ge c x e^{- x \sqrt{2\log2}},\label{eq: right tail lower bound exc}
		\end{equation}
		where $\sqrt{2 n} = \kappa L + x$.		
	\end{corollary}
	}

	A lower bound \corDB{for the cover time} in terms of excursions, \corDB{which will later lead to \eqref{eq: lower tail bound}} is given by the following proposition.
	\begin{proposition}
		\label{prop: Lower Bound cover}There is a constant $c$ such for
		all $x\ge0$,
		\begin{equation}
		\liminf_{L\to\infty}
		%\mathbb{P}
		\corDB{\mathbb{P}}\left(\min_{\corO{y\in \ll_L}}T_{L}^{y,n}=0\right)\ge\corO{1-} ce^{-cx},\label{eq: lower bound in terms of excursions}
		\end{equation}
		where
		\corO{
			%\[
			$\sqrt{2n}=\kappa L-x$.}
		% \sqrt{2\log2}L-\frac{1}{\sqrt{2\log2}}\log L-x.
		%\]
	\end{proposition}
	%\begin{remark}
	%\corDB{ By analogy to Branching Random Walk once excepts the true size of the left tail to be double exponential of form %$\exp( -c e^{cx} )$ for appropriate constants.}
	%\end{remark}
	\begin{proof}
		%Let 
		%\[
		%\kappa=\sqrt{2\log2}-\frac{\log L}{\sqrt{2\log2L}}\text{ so that }\sqrt{2n}=\kappa L-x.
		%\]
		\corO{We will fix $k=k(x)$ below. } \corDB{ From \eqref{eq: subtree trav count compat} it follows that}% In particular
		\begin{equation}
		\corO{T_{k}^{y,n}}\le m\text{ and }T_{L-k}^{y,k,m}=0\implies 
		\corO{T_{L}^{y,n}}=0.\label{eq: split tree property}
		\end{equation}
		We consider the processes $\left(T_{k+l}^{y,k,m}\right)_{l=0,\ldots,\bar{L}}$
		where $\bar{L}=L-k$. Using (\ref{eq: tail bound}) we have for any $k,a,L\ge1$ that
		\[
		%\mathbb{P}
		\corDB{\mathbb{P}}\left(\max_{y \in \ll_L }\sqrt{2T_{k}^{y,n}}>\left(\kappa L-x\right)+\kappa k+a\right)\le2^{k}\corO{e^{-\frac{\left(\kappa k+a\right)^{2}}{2k}}\le 
			c \corD{L^{k/L}} e^{-\kappa 
				a}},
		\]
		\corDB{by a union bound over the $2^k$ vertices in level $k$ of the tree, which is justified by \eqref{eq: reduced complexity}.} 
		Setting $k=cx$ and $a=x-2\kappa k-\corO{\frac{k}{c}\frac{\log L}{\bar L}}$
		one has $\bar{\kappa}\bar{L}\ge\corO{\kappa L-x}+\kappa \corO{k}+a$
		and $a\ge cx$ for small enough $c$ and large enough $L$, so that
		\begin{equation}
			\corD{\limsup_{L\to\infty}}\, 
			%\mathbb{P}
			\corDB{\mathbb{P}}\left(\max_{\corO{y\in \ll_L}}\sqrt{2T_{k}^{y,n}}>\bar{\kappa}\bar{L}\right)\le \corO{ce^{-cx}}.\label{eq: top tree}
		\end{equation}
		Setting $\sqrt{2 m} = \bar{\kappa} \bar{L}$ and using Lemma \ref{lem: subtree lower bound} with $0$ in place of $x$ it follows that
		\begin{equation}
			\liminf_{L\to\infty}
			%\mathbb{P}
			\corDB{\mathbb{P}}\left( \inf_{y \in \ll_L} T_{L}^{y,k,m}=0 \right) \ge \frac{2^k}{2^k + c} \ge 1 - c2^{-k}.\label{eq: to prove lower subtree endproof}
		\end{equation}
		\corDB{Together with \eqref{eq: top tree} this implies
		$$
			\begin{array}{l}
				\liminf_{L \to \infty}\mathbb{P}\left( \left\{\forall y \in \ll_L \mbox{ } T^{y,k}_k \le m  \right\} \cap \left\{ \exists y \in \ll_L \mbox{ s.t. } T_{L}^{y,k,m}=0 \right\}  \right) \\
				\ge 1 - c2^{-k} - c e^{-cx},
			\end{array}
		$$
		which implies \eqref{eq: lower bound in terms of excursions}, because of (\ref{eq: split tree property}) and since $k=cx$.}
	\end{proof}

	We now relate excursion time to real time to derive Theorem \ref{thm: tree cover main}
	from Proposition \ref{prop: Upper Bound cover} and Proposition \ref{prop: Lower Bound cover}.
	Note that 
	\[
	D_{n}=S_{1}+S_{2}+\ldots+S_{n},
	\]
	where $S_{i}$ is the length of the $i$-th excursion from the root.
	By the strong Markov property the $S_{i},i\ge1,$ are iid. Elementary
	1D random walk computations show that $
	%\mathbb{E}
	\corDB{\mathbb{E}}\left(S_{i}\right)=2^{L+2}-2$
	and by Khasminskii's lemma (a consequence of Kac's moment formula,
	see (6) \corDBa{\cite{FitzsimmonsKac}} we have
	$
	%\mathbb{E}
	\corDB{\mathbb{E}}\left(S_{i}^{k}\right)\le k!
	%\mathbb{E}
	\corDB{\mathbb{E}}\left(S_{i}\right)^{k}$.
	Thus $
	%\mathbb{E}
	\corDB{\mathbb{E}}\left(S_{i}^{2}\right)\le c2^{2L}<\infty$, so that
	by the central limit theorem $\frac{D_{n}-n
		%\mathbb{E}
		\corDB{\mathbb{E}}
		\left(S_{1}\right)}{\sqrt{n}\sqrt{\text{Var}\left(S_{1}\right)}}$
	converges to a normal distribution as $n\to\infty$. Also $
	%\mathbb{E}
	\corDB{\mathbb{E}}\left(S_{i}^{3}\right)\le c2^{3L}\le c\left(\sqrt{\text{Var}\left(S_{i}\right)}\right)^{2}$
	(it is easily seen that $\text{Var}\left(S_{i}\right)\ge c2^{2L}$),
	so that by the Berry-Essen theorem in addition
	\begin{equation}
	\sup_{x\in\mathbb{R}}\left|
	%\mathbb{P}
	\corDB{\mathbb{P}}\left(\frac{D_{n}-n
		%\mathbb{E}
		\corDB{\mathbb{E}}\left(S_{1}\right)}{\sqrt{n}\sqrt{\text{Var}\left(S_{1}\right)}}\le x\right)-\Phi\left(x\right)\right|\le\frac{c}{\sqrt{n}},\label{eq: Berry Esseen}
	\end{equation}
	\corDBa{uniformly in $L$.} We now prove the estimate Theorem \ref{thm: tree cover main} for
	the cover time.
	
	\begin{proof}[Proof of Theorem \ref{thm: tree cover main}]
		For the upper bound, let
		\[
		\sqrt{2n}=\sqrt{2\log2}L-\frac{1}{\sqrt{2\log2}}\log L+\frac{x}{2}.
		\]
		Because of (\ref{eq: hitting time traversal count connection}) we
		have
		\begin{equation}
		\begin{array}{l}
		%\mathbb{P}
		\corDB{\mathbb{P}}\left(C_{L}>2^{L+1}\left(\sqrt{2\log2}L-\frac{1}{\sqrt{2\log2}}\log L+x\right)^{2}\right)\\
		\le
		%\mathbb{P}
		\corDB{\mathbb{P}}\left(D_{n}>2^{L+1}\left(\sqrt{2\log2}L-\frac{1}{\sqrt{2\log2}}\log L+x\right)^{2}\right)+
		%\mathbb{P}
		\corDB{\mathbb{P}}\left(\min_{y\corO{\in \ll_L}
			%\text{ leaf}
		}T_{L}^{y,n}=0\right).
		\end{array}\label{eq: proof of upper bound split}
		\end{equation}
		Since 
		\begin{eqnarray*}	
		&&2^{L+1}\left(\sqrt{2\log2}L-\frac{1}{\sqrt{2\log2}}\log L+x\right)^{2}\\
		&&\quad 
		=2^{L+2}L\left(\left(\log2\right)L-\log L + \sqrt{2 \log 2}x+o\left(1\right)\right),
	      \end{eqnarray*}	
		and
		\[
		n=L\left(\left(\log2\right)L-\log L+\sqrt{2 \log 2}\frac{x}{2}+o\left(1\right)\right),
		\]
		we have from (\ref{eq: Berry Esseen}) that 
		\[
		%\mathbb{P}
		\corDB{\mathbb{P}}\left(D_{n}>2^{L+1}\left(\sqrt{2\log2}L-\frac{1}{\sqrt{2\log2}}\log L+\sqrt{2 \log 2}x\right)^{2}\right)\le ce^{-cx^{2}}+\frac{c}{\sqrt{n}},
		\]
		and combining this with Proposition \ref{prop: Upper Bound cover} we get the claim
		(\ref{eq: upper tail bound}) from (\ref{eq: proof of upper bound split}). Since the right-hand side tends to zero as $x\to\infty$ this proves
		the upper bound \corDB{on the right tail \eqref{eq: upper tail bound}. The lower bound \eqref{eq: upper tail lower bound} on the right tail follows similarly from \eqref{eq: right tail lower bound exc} and \eqref{eq: Berry Esseen}.} %of Theorem \ref{thm: tree cover main}.
		
		For the bound \corDB{\eqref{eq: lower tail bound} on the left tail} let
		\[
		\sqrt{2n}=\sqrt{2\log2}L-\frac{1}{\sqrt{2\log2}}\log L-\frac{x}{2}.
		\]
		We have that
		\[
		\begin{array}{l}
		%\mathbb{P}
		\corDB{\mathbb{P}}\left(C_{L} \le 2^{L+1}\left(\sqrt{2\log2}L-\frac{1}{\sqrt{2\log2}}\log L-x\right)^{2}\right)\\
		\le
		%\mathbb{P}
		\corDB{\mathbb{P}}\left(D_{n}\le2^{L+1}\left(\sqrt{2\log2}L-\frac{1}{\sqrt{2\log2}}\log L-x\right)^{2}\right)+
		%\mathbb{P}
		\corDB{\mathbb{P}}\left(\min_{y\corO{\in \ll_L}
			%\text{ leaf}
		}T_{L}^{y,n}>0\right).
		\end{array}
		\]
		Similarily to above (\ref{eq: Berry Esseen}) implies that is at most
		$ce^{-x^{2}}$, and the second term is at most $ce^{-cx}$ by Proposition \ref{prop: Lower Bound cover},
		which implies (\ref{eq: lower tail bound}).
	\end{proof}

\bibliographystyle{plain}

\bigskip
\noindent
\begin{tabular}{lll} & David Belius   \\
%& Department of
%Mathematics  \\
& Institute of Mathematics\\
&University of Z\"{u}rich\\
&CH-8057 Z\"{u}rich, Switzerland\\
& david.belius@cantab.net\\ 
& &\\
& & \\
& Jay Rosen\\
& Department of Mathematics\\
&  College of Staten Island, CUNY\\
& Staten Island, NY 10314 \\
& jrosen30@optimum.net\\
& &\\
& & \\
 & Ofer Zeitouni\\
& Faculty of Mathematics, 
 Weitzmann Institute and\\
 &Courant Institute, NYU\\
& Rehovot 76100, Israel and NYC, NY 10012 \\
& ofer.zeitouni@weizmann.ac.il
\end{tabular}

\end{document}